\documentclass[a4paper,10pt]{amsart}
\usepackage{amssymb}
\usepackage{cmap}
\usepackage{tikz-cd}
\usepackage[pdfdisplaydoctitle=true,
            colorlinks=true,
            urlcolor=blue,
            citecolor=blue,
            linkcolor=blue,
            pdfstartview=FitH,
            pdfpagemode= UseNone,
            bookmarksnumbered=true,
            backref=page]{hyperref}
\usepackage{hyperxmp}
\usepackage{todonotes}
\usepackage[all]{xy}
\usepackage{ulem}
\usepackage{cancel}

\theoremstyle{plain}
\newtheorem{thm}{Theorem}[section]
\newtheorem*{thm*}{Theorem}

\newtheorem{cor}[thm]{Corollary}
\newtheorem{lem}[thm]{Lemma}
\newtheorem{prop}[thm]{Proposition}

\theoremstyle{definition}

\theoremstyle{definition}
\newtheorem{rem}[thm]{Remark}

\newtheorem{defi}[thm]{Definition}

\newtheorem{expl}[thm]{Example}
\newtheorem*{ass*}{Assumption}

\numberwithin{equation}{section}

\newcommand{\NN}{\mathbb{N}} %
\newcommand{\ZZ}{\mathbb{Z}} %
\newcommand{\RR}{\mathbb{R}} %
\newcommand{\CC}{\mathbb{C}} %

\newcommand{\id}{\mathrm{id}}           %
\let\on=\operatorname

\newcommand{\Diff}{\mathrm{Diff}}       %
\newcommand{\biLip}{\mathrm{biLip}}

\let\on=\operatorname

\def\Diff{\on{Diff}}

\allowdisplaybreaks

\newcommand{\e}{\varepsilon}

\usepackage{todonotes}
\hypersetup{
    pdftitle={Sobolev metrics on spaces of manifold valued curves}, %
    pdfsubject={MSC 2010: 58B20, 58D10, 35G55, 35G60}, %
    pdfkeywords={Sobolev metrics, manifold valued curves}, %
    pdflang=EN %
    }

\begin{document}

\title [\ ] {An analytic approach to $P$-adic diffeomorphism group and  Teichm\"{u}ller theory}

\author{Yuxiu Lu}
\address{Department of Mathematics and statistics, Nanjing University of Science and Technology, 210094, P.R.China
 }
 \email{luyxnj@gmail.com}
\date{March 20, 2026}


\begin{abstract}
We consider a specific class of infinite dimensional $p$-adic Lie groups, i.e., a sort of diffeomorphism groups on $p$-adic ball $\operatorname{Diff}^{\operatorname{an}}(B_\epsilon)$. It turns out that this group has a natural logarithmic structure that leads to a $p$-adic version of
Teichm\"{u}ller theory on diffeomorphism groups, which also presents some remarkable hydrodynamic
facets. We further apply this framework to Mochizuki's $p$-adic Teichm\"{u}ller theory and Inter-universal Teichm\"{u}ller theory (IUT), and give a new reformulation of IUT as a Teichm\"{u}ller theory on automorphisms of two-dimensional group schemes.     
 \end{abstract}

\maketitle

\setcounter{tocdepth}{2}
\tableofcontents

 \section{introduction}
In his series of papers~\cite{mochizuki12},
Mochizuki constructed a theory called inter-universal Teichm\"{u}ller theory (IUT) to provide a proof for various outstanding conjectures in Diophantine geometry. It can be seen as an arithmetic version of Teichm\"{u}ller theory for number fields equipped with an elliptic curve, which is a continuation of his previous work in anabelian geometry and 
$p$-adic Teichm\"{u}ller theory. The previous approaches are purely algebraic and focuse on the \'etale fundamental group of  arithmetic variety. However, this \'etale-like picture alone is totally discrete and lacks the necessary analytic structure. The groundbreaking idea of Mochizuki is to introduce 
an enhancement of the \'etale-like picture, i.e., the Frobenius-like picture, that further gives rise to a log-structure.

Although Mochizuki's treatment for this topic is insightful and enlightening in many aspects, it involves many newly defined objects that are not easily approachable. Hence a natural question arises that whether one can find a suitable category where many of these new terminology  and definitions will be understood as some down-to-earth notions. Foremost we need to understand what is a Frobenius-like picture and how it comes.  

In recent paper~\cite{Lu2025}, the present author found that
there exists a $\RR$-version Teichm\"{u}ller theory on infinite dimensional Fr\'echet Lie group $\operatorname{Diff}_{-\infty}(\RR)$, which is a diffeomorphism group of the identity component. In this Fr\'echet Lie group endowed with right invariant $\dot W^{1,r}$  Sobolev metric, the $r$-th root map becomes a logarithmic function when $r$ goes to infinity, which procedure is denoted by an $L^r$-approximation. When we replace the real line with a $p$-adic field, it is quite remarkable that it is the power of Frobenius morphisms that turns into a logarithmic function at infinity; see Corollary~\ref{cor:loglim}. Therefore, the next task is to find a convenient diffeomorphism group on a $p$-adic field and a corresponding Teichm\"{u}ller theory parallel to that on $\operatorname{Diff}_{-\infty}(\RR)$. For this purpose, we introduce a $p$-adic diffeomorphism group $\operatorname{Diff}^{\operatorname{an}}(B_\epsilon)$ on a $p$-adic ball $B_\epsilon$, and consider its various differential geometric properties, such as its Lie algebra, Bers embedding, etc.

This gives us insight into the questions of what a Frobenius-like picture is in its essence. 
In fact, it can be seen as an infinitesimal approximation of a log-structure, and the Schwarz derivative as something derived from a 
log-structure can be approximated by  Frobenius morphisms; see subsection~\ref{subsec:schwarzder}. Moreover, the Frobenius-like picture is strictly related to hydrodynamics of the diffeomorphism group $\operatorname{Diff}^{\operatorname{an}}(B_\epsilon)$, which is more explicitly displayed in real line case of $\operatorname{Diff}_{-\infty}(\RR)$; for more information about differential geometric approach to hydrodynamics, see e.g.,~\cite{arnold1966geometrie, Arnold98}. Therefore, the Frobenius-like picture introduces a necessary local analytic structure on the \'etale-like picture a priori.  

All of these above discussions consist of the first part of the present article. We aim to provide a new approach to understanding the Frobenius-like picture in a purely analytic way. In the second part, we will mainly apply this method to Mochizuki's $p$-adic Teichm\"{u}ller theory and inter-universe Teichm\"{u}ller theory by considering its interplay with the \'etale-like picture.

The $p$-adic Teichm\"{u}ller theory is a generalization of classical Teichm\"{u}ller theory on a hyperbolic curve $C$ over an arbitrary positive characteristic perfect field. Its main object is an integral Frobenius indigenous bundle, which was originally introduced by Gunning~\cite{gunning1967} as an affine (resp. projective)  flat coordinate bundle on a Riemann surface. It was shown that the existence of a Fuchsian uniformization is equivalent to that of a canonical indigenous bundle on Riemann surface; this idea was further developed by Mochizuki to study the integrality condition on the $p$-adic indigenous bundle. Unlike Mochizuki's algebraic prospective, our approach will focus on the diff-group theoretic aspects of this object, that is, the automorphism group of 
the indigenous bundle.

In their paper~\cite{abbati1989}, Abbati et. al. showed that for any principal bundle $P$, the  automorphism group  $\operatorname{Aut}(P)$ can be characterized as an exact sequence
\begin{equation*}
1\to \operatorname{Gau}(P)\to \operatorname{Aut}(P)\to \operatorname{Diff}_{\operatorname{id}}(M)\to1,     
\end{equation*}
where $\operatorname{Diff}_{\operatorname{id}}(M)$ is an open subgroup of $M$ containing the connected component of the identity, and $\operatorname{Gau}(P)$ denotes gauge group of $P$. Henceforth, if the fibre of indigenous bundle is a group, then we can give a full characterization of $\operatorname{Aut}(P)$ in two directions: one  is the local analytic structure pertaining to $\operatorname{Diff}_{\operatorname{id}}(M)$, the other is the \'etale-like picture pertaining to the group transformations of fibres. This phenomenon is common in hydrodynamics, one case of interest is the space of symplectic structures realized as a principal bundle on 
the space of probability densities; see~\cite{Lu2025}. But this is not quite applicable to an indigenous bundle on a hyperbolic curve, since it corresponds to a projective structure and there is no group structure on the fibre $\mathbb P_K^1$. Thus we are only capable of considering in this case its local analytic properties, i.e., its Frobenius-like picture.

However, it is quite surprising that this interpretation appears to fit well into the pattern of inter-universal Teichm\"{u}ller theory (IUT). The main object of IUT is an elliptic curve $E$ over a number field $K$. Here  $E$ can be seen as an arithmetic surface, or a group scheme $\mathcal{E}$, which is a two-dimensional geometric object: the base space $R$, as the ring of integers of $K$, can be locally completed to a  complete discrete valuation ring; the special (generic) fiber, as an elliptic curve, has a canonical group structure. Therefore, it is an analogy of a principal bundle in algebraic geometry, and furthermore, the automorphism group of $\mathcal{E}$ can be split into two parts in an exact sequence, as in the case of a principal bundle. The first part is the local analytical structure pertaining to the diffeomorphism group of a complete ring $\operatorname{Diff}^{\operatorname{an}}(R_{\mathfrak p})$ which we already know well from the theory of $p$-adic diffeomorphism groups. The second part is the \'etale-like picture related to the group transformations acting on the elliptic curve $E$. Since the group structure on elliptic curve is too large and the automorphism group is too small (the order dividing $24$), we choose the group of $l$-torsion points on $E$ as is done in IUT.  This group can be evaluated in terms of a theory of $p$-adic theta functions over a Tate curve $E_q$, we will follow steps in~\cite{brown2003} to introduce this topic.  
Putting all of these above together, we obtain the so-called Hodge theatre, which consists of the log-links given by the local analytic structure, and the $\Theta$-links given by $l$-th roots of $p$-adic theta functions, with respect to the \'etale picture. Finally, we present a brief remark regarding the Diophantine 
results that were derived by IUT. Although we could not fully jusitify these results, we explain some of its fundamental ideas and how it gets stuck at some points.

\subsection*{The structure of the paper}
Section~\ref{sec:padicman}
gives some background information of $p$-adic manifolds, differential forms and Serre's classification theorem on 
$p$-adic compact manifolds.
Section~\ref{sec:padicdiffgroup} introduces a specific $p$-adic diffeomorphism group $\operatorname{Diff}^{\operatorname{an}}(B_\epsilon)$, defines a 
$p$-adic version of
Teichm\"{u}ller theory on $\operatorname{Diff}^{\operatorname{an}}(B_\epsilon)$ by an approximation of logarithmic function from Frobenius morphisms. Section~\ref{sec:padicteichmuller}, as a transitional section,  concerns an exposition of the basic notions used in $p$-adic Teichm\"{u}ller theory, further links the study on automorphism groups of indigenous bundle to that on a principal bundle. We also construct a canonical Frobenius lifting, the key element in $p$-adic Teichm\"{u}ller theory, from the logarithmic structure given in the previous sections.
The final Section~\ref{sec:interuniverseteich}
is a reformulation of IUT by using a point of view from $p$-adic diffeomorphism groups. We will discuss fundamental notions from scheme theory, Tate curves
and theta functions, and consider the log-links pertaining to  log-structure on a $p$-adic completion of a number field, as well as the $\Theta$-links pertaining to the theta values on $l$-th torsion points on an elliptic curve.

 \section{$p$-adic analytic manifolds and differential forms}\label{sec:padicman}
Let $(K,|\cdot|)$ be a non-archimedean local field with an ultrametric. The notion of a $p$-adic analytic manifold $M$ over $K$ can be found in the books~\cite{schneider2011, serre1992}. We present some related notions in the following.
\subsection{Charts and differential forms}
Let $M$ be a Hausdorff topological space. An $n$-chart on $M$ is a pair $(U_\alpha, \phi_\alpha)$, where $U_\alpha$ is an open subset and $\phi_\alpha:U_\alpha\to K^n$ is a homeomorphism onto an open subset in $K^n$. In addition, two charts $(U_\alpha, \phi_\alpha)$ and $(U_\beta, \phi_\beta)$ are compatible if the transition map 
\begin{align*}
 g_{\beta\alpha}:= \phi_\beta\circ \phi_\alpha^{-1}: \phi_\alpha(U_\alpha)\to\phi_\beta(U_\beta)
\end{align*}
is locally analytic. An atlas of $M$ is a family of compatible charts $\mathcal A=\{(U_\alpha,\phi_\alpha):\alpha\in I\}$ such that the set $\{U_\alpha:\alpha\in I\}$ forms an open cover of $M$. Two atlases $\mathcal A$ and  $\mathcal B$ are called equivalent if $\mathcal A\cup\mathcal B $ is also an atlas of $M$. An equivalent class of atlases on $M$ is called a 
$p$-adic analytic structure on $M$. There is a unique maximal atlas in each analytic structure of $M$ with respect to inclusion maps; see e.g., Remark 7.2 in~\cite{schneider2011}. In conclusion, a $p$-adic analytic manifold is a pair $(M, \mathcal A)$ consisting of a Hausdorff space $M$ and an atlas $\mathcal A$ of $M$ given above.

Similar to the real case, we are able to define the differential forms on $M$; note that the name "differential" here no longer refers to smoothness, but rather local analyticity. 
Therefore, an analytic differential $n$-form $\omega$ on $p$-adic analytic manifold $M$ can be expressed locally in each chart $(U_\alpha,\phi_\alpha)$ as 
\begin{align*}
\omega_\alpha&=f_\alpha(\phi_\alpha(x)) dx\\&=f_\alpha(x_1,\cdots,x_n) dx_1\wedge\cdots\wedge dx_n    
\end{align*}
 with $f_\alpha:\phi_\alpha(U_\alpha)\to K$ a locally analytic map. 
 
 On the other hand, since $K$ (e.g., the field of $p$-adic numbers
 $\mathbb Q_p$) is a locally compact topological group with respect to addition, there exists a unique additive Haar measure $|dx|$ by Weil theorem. More precisely, it is a positive measure invariant under shifts, i.e., $|d(x+a)|=|dx|$, for any $a\in \mathbb Q_p$. If the measure $|dx|$ is normalized by the equality 
 \begin{equation}\label{normalizedhaar}
 \int_{\mathbb Z_p} |dx|=1,   
 \end{equation}
then $|dx|$ is uniquely determined. Furthermore, the right multiplication of $a\in\mathbb Q_p^*$ acts on $|dx|$ by 
\begin{equation}
|d(xa)|=|a|_p |dx|,    
\end{equation}
where $|\cdot|_p$ is the canonical ultrametric on $\mathbb Q_p$; for a proof, see e.g., Proposition 3.2.1 in~\cite{Albeverio2010}. The Haar measure on $\mathbb Q_p$ can be easily extended to the one on $\mathbb Q_p^n$. For any linear isomorphism $A:\mathbb Q_p^n\to\mathbb Q_p^n$, we have similar transition laws for both addition 
\begin{equation}
|d(x+a)|=|dx|, \; a\in \mathbb Q_p^n,   
\end{equation}
and multiplication
\begin{equation}
|d(Ax)|=|\operatorname{det}A|_p|dx|, \; a\in \mathbb Q_p^n.   
\end{equation}
The formula of change of variables still holds in the $p$-adic setting.
\begin{prop}\label{changevariable}
If $y=\phi(x)$ is an analytic diffeomorphism of a closed and open set $A\subset \mathbb Q_p^n$ onto a closed and open set $A'\subset \mathbb Q_p^n$, and the Jacobian $\det(D\phi)(x)\neq0$ for any $x\in A$. Then for any integrable function $f:A\to\mathbb C$, we have 
\begin{align*}
\int_A f(y)|dy|=\int_{A'}f(\phi(x))|\operatorname{Det}(D\phi)(x)|_p|dx|.    
\end{align*}
\end{prop}
\begin{proof}
See Theorem~3.4.1 and Theorem~3.4.2 of the book~\cite{Albeverio2010}.    
\end{proof}
It turns out that every $n$-form $\omega$ on a $p$-adic manifold $M$ generates a Haar measure $\mu_\omega=|\omega|$ on $M\setminus V(\omega)$, where $V(\omega)$ is the vanishing set of $\omega$. The measure $\mu_\omega$ on a subset $A\subset U_\alpha$ can then be written as 
\begin{align*}
\mu_\omega(A)=\int_A|\omega|=\int_{\phi_\alpha(A)}  |f_\alpha(x)||dx|.  
\end{align*}

\subsection{Compact $p$-adic manifolds and Serre's theorems}
Compared to manifolds modeled on Euclidean spaces, the structure of compact manifolds on $p$-adic fields becomes much simpler, thanks to the discrete nature of $p$-adic numbers. Indeed, according to a theorem of Bourbaki, all the $p$-adic analytic manifolds are the disjoint unions of some $p$-adic balls. 

In his paper~\cite{serre1965}, Serre characterizes an arbitrary compact $p$-adic manifolds by its invariant. In the following, we state these short but remarkable results. The proofs can be found in~\cite{igusa2002, serre1965}. 
\begin{thm}\label{serrethm1}
Every $n$-dimensional compact manifold $M$ is isomorphic to the $r$-copies ($r\geq1$) of $p$-adic balls $ ^r{(B_\epsilon(0))}$, where the ball $B_\epsilon(a):=\{x:|x-a|_p<\epsilon=p^\gamma\}$. Furthermore, $^r(B_\epsilon(0))$ and $^{r'}(B_\epsilon(0))$ are isomorphic if and only if $r\equiv r'\operatorname{mod}(q-1)$.    
\end{thm}
It is quite unexpected that the integration of any non-vanishing $n$-differential form $\omega$ on $M$ gives rise to an invariant modulo $q-1$. 
\begin{thm}\label{serrethm2}
Let $q=p^f$ such that $|a|_v=q^{v(a)}$. There exists a non-vanishing analytic differential form $\omega$
on $M$, all such $\omega$ have the integration $\int_M |\omega|=N/q^m$, for $m\in\mathbb N$ and $N\in\mathbb N^+$. This integration modulo $q-1$ 
\begin{equation}
i(M):=\mu_\omega(M)=\int_M|\omega|\equiv N \; \; \operatorname{mod}(q-1),    
\end{equation}
depends only on $M$, and $M$ is $K$-bianalytic to $(B_\gamma(0))^{i(M)}$.
\end{thm}
This invariant $i(M)$ in essence is similar to the situation of a characteristic class or number for real/complex differential geometry, the change of underlying differential structure $\omega$ would not alter the invariant.  
\begin{expl}[Theorem~4.1, \cite{bradley2025}]
 Let $E$ be an elliptic curve over $K$ with the characteristic $p\neq2,3$. If it has good reduction, then its Weierstrass equation 
 \begin{equation}
 \overline E: \ y^2=x^3-\overline{27}\overline{c_4}x- \overline{54}\overline{c_6}   
 \end{equation}
 defined over the residue field $k=\mathcal O_k/\pi\mathcal O_k$ is non-singular. The $K$-rational points $E(K)$ forms a $1$-dimensional closed $p$-adic analytic submanifold of $\mathbb P^2(K)$. By Theorem 2.2.5 of~\cite{weil1982}, the Serre invariant is 
 \begin{equation}
i(E(K))=|\overline E(k)| \; \; \operatorname{mod}(q-1),     
 \end{equation}
 where  $q=|k|$ is the cardinality of the residue field $k$, and $|\overline E(k)|$ is the number of points in $\overline E(k)$.
\end{expl}

 \section{$p$-adic Lie group and Teichm\"{u}ller theory}\label{sec:padicdiffgroup}
We recall the notion of a $p$-adic Lie group, for which the smoothness of the multiplication map is replaced by local analyticity.  
\begin{defi}
A $p$-adic Lie group (possibly of infinite dimension) over local field $K$ is a manifold endowed with the structure of a group such that the multiplication map 
$m: G\times G\to G \; \ \text{via}: (g,h)\to gh  $ 
is locally analytic.
\end{defi}
There are several examples of $p$-adic Lie groups of finite dimension, such as the ball $(B_\epsilon(0),+)$ with addition, and the ball $(B_\epsilon(1),\cdot)$ with multiplication; see~\cite{schneider2011}. The infinite dimensional case, or more specifically, the diffeomorphism group of $p$-adic manifolds is of interest in many aspects. We have seen from Serre's theorem~\ref{serrethm1} that a compact $p$-adic manifold is essentially an $r$-copy of $p$-adic balls, so it suffices to investigate the diffeomorphism group of $^r(B_\epsilon(0))$. Note that $^r(B_\epsilon(0))\to^r\mathbb Q_p^n$ for $\epsilon\to\infty$, hence the diffeomorphism group of $\mathbb Q_p^n$ can be viewed as a limit of the diffeomorphism group of $^r\mathbb (B_\epsilon(0))$. For all of these reasons, we restrict ourselves to the study on diffeomorphism group of $B_\epsilon(0)$.

\subsection{The convenient  diffeomorphism group $\operatorname{Diff}^{\operatorname{an}}(B_\epsilon)$}
In~\cite{Lu2025}, the present author showed that the $L^p$-geometry on $\Diff_{-\infty}(\RR)$ gives rise to a sort of Teichm\"{u}ller theory on real line. It is natural to question whether there exists such formulation in a $p$-adic  field. The main ingredients consist of an $r$-th root map and a diffeomorphism group whose elements are sufficiently close to the identity. Therefore, we consider the analytic diffeomorphism group $\operatorname{Diff}^{\operatorname{an}}(B_\epsilon)$ in a $p$-adic ball $B_\epsilon:=B_\epsilon(0)$. 
\begin{align*}
\operatorname{Diff}^{1}(B_\epsilon)&:=\left\{\varphi=\id+f: f\in C^{1}(B_\epsilon),\; ||Df||_{\operatorname{lip}}\leq L<1,\text{ and } f(0)=0 \right\};\\    
\operatorname{Diff}^{\operatorname{an}}(B_\epsilon)&:=\left\{\varphi=\id+f: f\in C^{\operatorname{an}}(B_\epsilon),\; ||Df||_{\operatorname{lip}}\leq L<1,\text{ and } f(0)=0 \right\},    
\end{align*}
where $C^{1}(B_\epsilon)$ denotes the space of  all strictly differentiable  functions $f$, its subspace $C^{\operatorname{an}}(B_\epsilon)$ denotes the space of all locally analytic functions $f$, and $||Df||_{\operatorname{lip}}$ denotes the Lipschitz constant of the differential $Df$
\begin{equation}
||Df||_{\operatorname{lip}}:=\operatorname{sup}_{v_1\neq v_2\in B_\epsilon,x\in B_\epsilon}\frac{|D_xf(v_1)-D_xf(v_2)|}{|v_1-v_2|}.    
\end{equation}
In general, a locally analytic function must be strictly differentiable. This is why we introduce this notion, and it is useful in showing properties such as local invertibility; see~\cite{schneider2011}.

\begin{defi}
Let $K$ be a non-archimedean local field, and let $V$ and $W$ be two normed $K$-vector spaces with $U$ an open subset of $V$. The map $f:U\to W$ is called differentiable at $v_0\in U$ if there exists a continuous linear map $D_{v_0}f:V\to W$ such that for any $\epsilon>0$, there is an open neighborhood $U_\epsilon\subset U$ of $v_0$ such that for any $v\in U_\epsilon$,
\begin{equation}\label{strictdiff1}
|f(v)-f(v_0)-D_{v_0}f(v-v_0)|\leq \epsilon|v-v_0|.    
\end{equation} 
$f$ is called
strictly differentiable at $v_0\in U$ if there exists a continuous linear map $D_{v_0}f:V\to W$ such that for any $\epsilon>0$, there is an open neighborhood $U_\epsilon\subset U$ of $v_0$ such that for any $v_0,v_1\in U_\epsilon$,
\begin{equation}\label{strictdiff2}
|f(v_1)-f(v_2)-D_{v_0}f(v_1-v_2)|\leq \epsilon|v_1-v_2|.    
\end{equation}
\end{defi}
\begin{rem}
The strict differentiability is a locally constant property in essence. If we take any other base point $v_0'\in U_\epsilon$ and $D_{v_0'}f:=D_{v_0}f$, then condition~\ref{strictdiff2}
still holds. Moreover, the linear map $D_{v_0}f$ is uniquely determined: if we have another linear map $\widetilde D_{v_0}f$ satisfying condition~\ref{strictdiff2}, then a straightforward calculation shows
\begin{align*}
||\widetilde D_{v_0}f(v_1-v_2)-D_{v_0}f(v_1-v_2)||\leq\epsilon||v_1-v_2||,
\end{align*}
for arbitrarily small $\epsilon>0$. It is also clear  that  all the norms of differential are smaller than the Lipschitz constant, i.e., $||D_{v_0}f||\leq||Df||_{\operatorname{lip}}$.   
\end{rem}
It is not hard to see that the addition and composition of two strictly differentiable functions are still strictly differentiable.
\begin{prop}\label{strictdiff}
If $f$ (resp. $g$) are strictly differentiable at $g(v_0)$ (resp. at $v_0$) with the operator norms $D_{g(v_0)}f$ (resp. $D_{v_0}g$) bounded, then $f\circ g$ are strictly differentiable at $v_0$.
\end{prop}
\begin{proof}
We only need to prove the composition is strictly differentiable. Since $f$ is strictly differentiable, we have
\begin{align*}
|f\circ g(v_1)- f\circ g(v_2)-D_{g(v_0)}f(g(v_1)-g(v_2))|\leq\epsilon |g(v_1)-g(v_2)|\leq \epsilon ||D_{v_0}g|||v_1-v_2|,   
\end{align*}
for any $g(v_1),g(v_2)\in U^f_\epsilon$ with $U^f_\epsilon$ an open neighborhood of $g(v_0)$. In addition, since $g$ is strictly differentiable at $v_0$, we have 
\begin{align*}
|D_{g(v_0)}f(g(v_1)-g(v_2))-D_{g(v_0)}f\circ D_{v_0}g(v_1-v_2)|\leq \epsilon'||D_{v_0}f|||v_1-v_2|,    
\end{align*}
for any $v_1,v_2\in U^g_{\epsilon'}$ with $U^g_{\epsilon'}$ an open neighborhood of $v_0$. Adding up these two inequalities, we finally get an inequality 
\begin{align*}
|f\circ g(v_1)- f\circ g(v_2)-D_{g(v_0)}f\circ D_{v_0}g(v_1-v_2)|\leq (\epsilon||D_{v_0}g||+\epsilon'||D_{v_0}f||)|v_1-v_2|,   
\end{align*}
for any $v_1,v_2\in U^g_{\epsilon'}\cap g^{-1}(U^f_{\epsilon})$ with $U^g_{\epsilon'}\cap g^{-1}(U^f_{\epsilon})$ an open neighborhood of $v_0$. The norms $||D_{g(v_0)}f||$ and $||D_{v_0}g||$ are bounded, so the scalar $\epsilon||D_{v_0}g||+\epsilon'||D_{v_0}f||$ can be arbitrarily close to zero as $\epsilon,\epsilon'\to 0^+$.
\end{proof}

Next we show that an arbitrary element of $\operatorname{Diff}^{\operatorname{an}}(B_\epsilon)$ is indeed a diffeomorphism (i.e., bianalytic map) of $B_\epsilon$. It should be noted that the argument of proof is also applicable to $\operatorname{Diff}^{1}(B_\epsilon)$, so $\operatorname{Diff}^{1}(B_\epsilon)$ is also a group with the differential structure of lower regularity. First we state a lemma regarding the estimate on the norm of inverse of $\varphi\in \operatorname{Diff}^{\operatorname{an}}(B_\epsilon)$.
\begin{lem}\label{lemdiffgroup}
Let $\psi=\varphi^{-1}$ and $\varphi(0)=0$. Then it can be written as $\psi(y)=y+g(y)$ such that the operator norm $||D_0g||\leq L<1$.   
\end{lem}
\begin{proof}
From the identity $\varphi\circ \psi(y)=y$, we see that $g(y)=-f(y+g(y))$. Differentiating the equation with respect to $y$, we obtain
\begin{align*}
D_0g(y)=\frac{-D_0f(y+g(y))}{1+D_0f(y+g(y))}.
\end{align*}
Here we use the chain rule for $p$-adic analysis. Therefore, the conclusion follows by the inequality
$||D_0f||\leq L<1$.
\end{proof}
\begin{thm}\label{thmdifflie}
$\operatorname{Diff}^{\operatorname{an}}(B_\epsilon)$ (resp. $\operatorname{Diff}^{\operatorname{1}}(B_\epsilon)$) is an infinite dimensional Lie group (resp. group) consisting of diffeomorphisms of $B_\epsilon$.    
\end{thm}
\begin{proof}
We first show that any $\varphi\in \operatorname{Diff}^{\operatorname{an}}(B_\epsilon)$ is a bijective onto $B_\epsilon$. Indeed, $\varphi$ is an injective from the following direct calculation
\begin{align*}
|\varphi(v_1)-\varphi(v_2)|&=|(v_1-v_2)+f(v_1)-f(v_2)|\\&\geq (1-L)|v_1-v_2|.    
\end{align*}
The surjectivity of $\varphi$ is a consequence of Banach fixed point theorem. Let $w\in B_\epsilon$ be an arbitrary fixed vector. For any $v\in B_\epsilon$, we take $v'=w+v-\varphi(v)$. Then we compute
\begin{align*}
|v'|&\leq\operatorname{max}\left(|v'-v|, |v|\right)   \leq\operatorname{max}(|w-\varphi(v)|, \epsilon)\\&   \leq\operatorname{max}(|w-\varphi(0)|,|\varphi(v)-\varphi(0)|, \epsilon)\leq\operatorname{max}(|v|, \epsilon)\leq \epsilon
\end{align*}
Hence we may define a sequence $(v_n)_{n\geq0}$ in $B_\epsilon$ by
\begin{align*}
v_{n+1}:=w+v_n-\varphi(v_n).    
\end{align*}
It is clear from $||D_{v_0}f||\leq L<1$ that there exists $0<\epsilon<1$ such that 
\begin{align*}
||v_{n+1}-v_n||\leq\epsilon||v_n-v_{n-1}||.    
\end{align*}
It follows from this contraction of $(v_n)$ that it is a Cauchy sequence converging to a limit $v=\operatorname{lim}v_n$ in the closed ball $B_\epsilon$. By passing to the limit of the defining equation and using the continuity of $\varphi$, we obtain $\varphi(v)=w$. This concludes the proof of surjectivity.

Next we check the local diffeomorphism of $\varphi$. Without loss of generality, we assume that $f$ is strictly differentiable at $0$ with respect to a neighborhood $U$ of $0$.  For any $v_1,v_2\in B_\epsilon$, we have 
\begin{align*}
|\varphi(v_1)-\varphi(v_2)-(v_1-v_2)|&=|f(v_1)-f(v_2)|\\&\leq\operatorname{max}\left(|f(v_1)-f(v_2)-D_{0}f(v_1-v_2)|, |D_{0}f(v_1-v_2)|\right)   \\&\leq\operatorname{max}(\epsilon, L) |v_1-v_2|. 
\end{align*}
Then applying Lemma~4.2 in~\cite{schneider2011} and using the fact $\varphi(0)=0$, we deduce that $\varphi$ induces a homeomorphism
$U\to V=\varphi(U)$. We can further show that the inverse $\varphi^{-1}$ is contained in $\operatorname{Diff}^{\operatorname{an}}(B_\epsilon)$.
Indeed, $\varphi^{-1}$ is locally analytic by invertibility of power series. The strict differentiability of $\varphi^{-1}$ comes from almost the same line of proof in~\cite[Proposition~4.3]{schneider2011}. This assertion is thus concluded by Lemma~\ref{lemdiffgroup}.

It remains to show that the composition $\phi\circ\varphi\in \operatorname{Diff}^{\operatorname{an}}(B_\epsilon)$ for any $\varphi,\phi\in \operatorname{Diff}^{\operatorname{an}}(B_\epsilon)$. The explicit expressions $\varphi(x)=x+f(x)$ and $\phi(y)=y+g(y)$ yield 
\begin{align*}
\phi\circ\varphi(x)=x+\left(f(x)+ g(x+f(x))\right).   
\end{align*}
The local analyticity easily follows from this equality, and a straightforward $p$-adic analysis on the norm shows 
\begin{align*}
||D_0f+D_0 g(x+f(x))||&=||D_0f+D_0 g\circ (x+f(x))\cdot(1+D_0f)||
\\&\leq\operatorname{max}\left(||D_0f||, ||D_0 g\circ (x+f(x))||\cdot||(1+D_0f)||\right)\leq L.
\end{align*}
Putting all of these together,  we finish the proof with the strict differentiability shown in Proposition~\ref{strictdiff}.
\end{proof}
\begin{rem}\label{rem:lipmapnec}
One may pose the question whether we can replace the condition $||Df||_{\operatorname{lip}}<1$ by a weaker one $||Df||<1$. Unfortunately, this is not the case. To illustrate this, we take $f(x)=-x^p$. Then $||D_xf||=||-px^{p-1}||<1$ on $B_1$, but $\varphi(0)=\varphi(1)=0$, so $\varphi$ is no longer an injective on $B_1=\ZZ_p$. However, this does not mean that the boundness of differential prevents the function from being Lipschitz. If the differential $Df$ is bounded on $B_\epsilon$, i.e., there exists a constant $C>0$ such that $||D_xf||\leq C$ for all $x\in B_\epsilon$, then $f$ is a Lipschitz function with Lipschitz constant likely greater than $C$, e.g., $C<1$ in the above counter-example.
\end{rem}

As shown in~\cite{Lu2025} for the case of real line, one can enlarge $\operatorname{Diff}^{\operatorname{an}}(\mathbb Q_p)$ to include all affine transformations
\begin{align*}
\widetilde\Diff^{\operatorname{an}}(\mathbb Q_p)=\left\{\varphi=a\cdot\id+b+f: (\varphi-b)/a\in\Diff^{\operatorname{an}}(B_\epsilon), a\in\mathbb Q_p^\times,\ b\in B_\epsilon \right\},    
\end{align*}
where $\mathbb Q_p^\times:=\mathbb Q_p\setminus0$ is the set of invertible elements in 
$\mathbb Q_p$.
\begin{cor}\label{diffgrouplarge}
 The set $\widetilde\Diff^{\operatorname{an}}(\mathbb Q_p)$ defined above is indeed an infinite dimensional Lie group.   
\end{cor}
\begin{proof}
In addition to group actions regarding small perturbations of identity, we need to know how translation and  scaling act on  $p$-adic field $\mathbb Q_p$. For any invertible element $a$ in $\mathbb Q_p$, the scaling map $S(x)=ax$ is a diffeomorphism. So it suffices to consider a translation map $T(x)=x+b$. Note that the condition $\varphi(0)=0$ only appears in the proof of Theorem~\ref{thmdifflie} to show that $\varphi$ is a homeomorphism. However, translation by an element $b\in B_\epsilon$ would not affect the proof, since in that case, we actually have a homeomorphism $\varphi:\mathbb Q_p\to \mathbb Q_p$ by Lemma~4.2 in~\cite{schneider2011}. %It is a well-known fact that every point of a $p$-adic ball is its center, so we have $B_\epsilon(0)= B_\epsilon(b)$.
\end{proof}

\subsection{A remark on higher dimensional case}

Unlike the case of real line, it is not hard to extend the preceding arguments to higher dimensional $p$-adic balls. We only need to replace the $p$-adic ball in $K=\mathbb Q_p$ with that in $K^n$, and then all derivatives are replaced by Jacobian matrices. In addition, the Lemma~4.2 and Proposition~4.3  used in the proofs have already been proved in a general setting in~\cite{schneider2011}.
Therefore, we restate the main theorems for potential use.
\begin{thm}\label{thmdiffliencase}
Let $B_\epsilon$ be a  
$p$-adic ball in $\mathbb Q_p^n$.
$\operatorname{Diff}^{\operatorname{an}}(B_\epsilon)$ (resp. $\operatorname{Diff}^{\operatorname{1}}(B_\epsilon)$) is an infinite dimensional Lie group (resp. group).    
\end{thm}

\subsection{Profinite group $\operatorname{Diff}^{\operatorname{an}}(B_\epsilon)$}

We consider the modulo $p^m$-version of $\operatorname{Diff}^{\operatorname{an}}(B_\epsilon)$ on a modulo $p^m$-ball $\overline B_\epsilon$ for $\epsilon\leq1$,
\begin{equation}
\operatorname{Diff}_{m}^{\operatorname{an}}(\overline B_\epsilon):=\left\{\overline\varphi \; \operatorname{mod}(p^m):
\varphi\in \operatorname{Diff}^{\operatorname{an}}(B_\epsilon)
\right\}.    
\end{equation}
The modulo action preserves the group structure, though it eliminates local analyticity. 
\begin{prop}
$\operatorname{Diff}_{m}^{\operatorname{an}}(\overline B_\epsilon)$ is a finite group.  In particular,
$\operatorname{Diff}_{1}^{\operatorname{an}}(\overline B_\epsilon)=\{\operatorname{id}\}$.   
\end{prop}
\begin{proof}
First we have $\overline{\varphi\circ\phi}\equiv\overline\varphi\circ\overline\phi$. For any $\overline\varphi\in\operatorname{Diff}_{m}^{\operatorname{an}}(\overline B_\epsilon)$, there exists a diffeomorphism   $\phi\in\operatorname{Diff}^{\operatorname{an}}( B_\epsilon)$ such that $\varphi\circ\phi=\operatorname{id}$ and 
$\varphi\circ\phi=\operatorname{id}$. Then we have $\overline\varphi\circ\overline\phi\equiv\operatorname{id}$ and 
$\overline\varphi\circ\overline\phi\equiv\operatorname{id}$ modular $p^m$, so $\overline\phi$ is indeed the inverse. In addition, the finiteness of $\operatorname{Diff}_{m}^{\operatorname{an}}(\overline B_\epsilon)$ is a consequence of the finiteness of $\overline B_\epsilon$.
\end{proof}
Therefore, $\operatorname{Diff}^{\operatorname{an}}(B_\epsilon)$ is isomorphic to the inverse limit of an inverse system of discrete finite groups, i.e., a profinite group.

\subsection{The $p$-adic logarithmic function and  $p^m$-th power map}
First recall that in $p$-adic analysis, the exponential function $\operatorname{exp}:R\to\CC_p$ is defined by the following power series 
\begin{equation}
\operatorname{exp}(x)=\sum_{n=0}^\infty \frac{x^n}{n!},   
\end{equation}
where the range $\mathbb C_p$ is the field of complex $p$-adic numbers, and the domain is $R:=\{x\in\mathbb C_p: |x|_p<p^{-1/(p-1)}\}$ in which the power series converges in $p$-adic sense. Its inverse logarithmic function $\operatorname{log}:D\to\CC_p$ also makes sense
\begin{equation}
\operatorname{log}(1+x)=\sum_{n=1}^\infty (-1)^{n-1}\frac{x^n}{n},    
\end{equation}
where the domain $D$ of $1+x$ is $\{1+x: |x|_p<1\}$.
This definition of exponential function is similar to the usual real/complex case, in which there is another nice way to formulate the exponential function by using infinite product
\begin{align*}
\operatorname{exp}(x)=\operatorname{lim}_{n\to\infty} \left(1+\frac{x}{n}\right)^n.    
\end{align*}
This formulation can be extended to $p$-adic functions, though in a slightly different setting.
\begin{prop}\label{thm:explim}
In complex $p$-adic field $\CC_p$, we have the following equality
\begin{equation}\label{eq:exp}
\operatorname{exp}(x)=\operatorname{lim}_{m\to\infty} \left(1+p^m x\right)^{1/p^m}.     
\end{equation}
\end{prop}
\begin{proof}
We use the $p$-adic logarithmic function and its power series
\begin{equation}\label{logtaylorexp}
\frac{1}{p^m}\operatorname{log}(1+p^mx)=x+\sum_{n=2}^\infty (-1)^{n-1}\frac{p^{m(n-1)}x^n}{n},    
\end{equation}
where $x\in R$. It is easy to see that this power series converges as for a fixed $x$ and any $1>\e>0$, the norm $|p^m x|_p\leq\e$ for $m$ sufficiently large, and then the decay rate of $\e^{n-1}$ is much faster than $|n|_p$. Note that the summation of a double sequence would not change  
under rearrangement; see e.g., Lemma~3.3 in~\cite{schneider2011}. So once we have the following estimate for $m\to\infty$, 
\begin{equation}
\left|\frac{1}{p^m}\operatorname{log}(1+p^mx)-x\right |_p=\left|\sum_{n=2}^\infty (-1)^{n-1}\frac{p^{m(n-1)}x^n}{n}\right |_p\to0,    
\end{equation}
we know that $\frac{1}{p^m}\operatorname{log}(1+p^mx)$ converges to $x$ as $m\to\infty$, which concludes the proof.
\end{proof}
This gives rise to another form of logarithmic function in $p$-adic setting.
\begin{cor}\label{cor:loglim}
The $p$-adic logarithmic function  can be written as a limit in $D$
\begin{equation}\label{eq:limlog}
\operatorname{log}(x)=\operatorname{lim}_{m\to\infty} \frac{1}{p^m}\left(x^{p^m} -1\right).     
\end{equation} 
\end{cor}

\begin{proof}
First we show that the limit exists in this case.
Let $x\in D$, i.e., $x=1+z$, with $|z|_p<1$. Then the binomial expansion of $\frac{1}{p^m}\left(x^{p^m} -1\right)$ is given by
\begin{equation}\label{eq:logexpansion}
\frac{1}{p^m}\left(x^{p^m} -1\right)=z+\sum_{k=2}^{p^m}\binom{p^m}{k}\frac{z^k}{p^m}.   
\end{equation}
By using Legrendre's formula on the number of factors of $p$ in $k!$, we see that the number of factors of $1/p$ in $\binom{p^m}{k}/p^m$ is $v_p(k)$. Because for a fixed $|z|_p<1$, the decay rate of $|z|_p^k$ is much faster than $|k|_p$, the series~\ref{eq:logexpansion} converges as $k\to\infty$.

Next fix a point $x$ and let $y_m=\left(1+p^m x\right)^{1/p^m}$. The limit $y=\operatorname{lim}_{m\to\infty}y_m$ exists and equals $\operatorname{exp}(x)\neq0$ by formula~\ref{eq:exp}. Then we get
\begin{equation}
\operatorname{log}y=x=\frac{1}{p^m}\left(y_m^{p^m} -1\right)=\operatorname{lim}\frac{1}{p^m}\left(y_m^{p^m} -1\right)=\operatorname{lim}\frac{1}{p^m}\left(y^{p^m} -1\right).    
\end{equation}
\end{proof}

\begin{rem}
The formula~\ref{eq:limlog} is of interest in many aspects. On the one hand, it is one of the key ingredients of Teichm\"{u}ller theory on real line as shown in~\cite{Lu2025}. In this paper, the $r$-th root map given by  
\begin{equation}\label{eq:rrootreal}
\Phi_r: \begin{cases}
\left(\Diff_{-\infty}(\mathbb R), \dot W^{1,r}\right)  	&\to  \left(W^{\infty,1}(\mathbb R),L^r\right)\\ 
\varphi 						&\mapsto r\left(\varphi'^{1/r}-1\right)
\end{cases}
\end{equation}
induces a real-valued Schwarzian derivative (or $L^r$-Schwarzian derivative) and corresponding real Bers imbedding; here we use the notation $r$ to avoid possible misinterpretations with a prime number $p$. One may raise questions about why we change the exponent $1/r $ $(r\in[1,\infty])$ to the exponent $p^m$ in $p$-adic case. This is because
the discrete $p$-valuation on $p^m$ is in essence an inversion,
i.e., $|p^m|_p=p^{-m}$.

On the other hand, the polynomial map
\begin{equation}
\Phi_m:
\CC_p\to \CC_p \ \ \text{via}:
x \to						 \frac{1}{p^m}\left(x^{p^m}-1\right),
\end{equation}
and its limit for $m\to\infty$
\begin{equation}
\Phi_\infty:
D\to R \ \ \text{via}:
x \to						 \operatorname{log}(x)
\end{equation}
enlightens us of its strict relationship with Mochizuki's $p$-adic Teichm\"{u}ller theory; e.g.,~\cite{mochizuki96,mochizuki99}. Indeed, the original $r$-th root map arises from a homogeneous right invariant Sobolev metric $\dot W^{1,r}$ on a specific group of diffeomorphisms of real line $\operatorname{Diff}_{-\infty}(\RR)$, which is unexpectedly related to the projective structure on $\operatorname{Diff}_{-\infty}(\RR)$ or on $\operatorname{Prob}(M)$ for compact manifold case; see~\cite{Lu2025}. This coincides with the nature of projective structure with respect to an indigenous bundle over a  hyperbolic curve. In addition, $\Phi_m$ is essentially a composition of Frobenius morphisms on $p$-adic fields, and such actions are cornerstones for $p$-adic Teichm\"{u}ller theory. In his paper~\cite{mochizuki96}, Mochizuki further derives a logarithmic bundle and its connections from Frobenius morphisms, this is exactly what we do by pushing $\Phi_m$ to its limit, a logarithmic function. In consideration of all these above, it is likely that there are some intrinsic correspondences between our analytic approaches and Mochizuki's algebraic ones. We will discuss this analogy in the following sections.
\end{rem} 

In $D_a:=\{a+z: |z|_p<1\}$ for $a\in\{1,\cdots,p-1\}$, we introduce a generalization of  logarithmic function using the limit~\ref{eq:limlog}.
\begin{defi}
The generalized logarithmic function in $D_a:=\{a+z: |z|_p<1\}$
is defined by 
\begin{equation}\label{eq:genlog}
\widetilde{\operatorname{log}}(x):= \operatorname{lim}_{m\to\infty}\frac{1}{p^m}\left(x^{p^m} -a^{p^m}\right).    
\end{equation}   
\end{defi}
\begin{rem}
Note that this definition makes sense, as in  positive real line we have for any $x,a>0$,
\begin{align*}
\operatorname{log}(x)-\operatorname{log}(a)= \operatorname{lim}_{r\to\infty}r\left(x^{1/r} -a^{1/r}\right).    
\end{align*}
\end{rem}
One has a nice characterization of the generalized logarithmic function using Teichm\"{u}ller character $\omega$. 

\begin{prop}
$\widetilde{\operatorname{log}}(x)=\omega(\tilde a)\operatorname{log}(x/a)$.      
\end{prop}
\begin{proof}
The proof is simply to apply Proposition~\ref{teichmullerchar} to equation~\ref{eq:genlog}.    
\end{proof}

This generalization can be extended in two directions. An arbitrary $x\in\mathbb Q_p^\times$ can be written uniquely as $x=p^sag$, for $s\in\ZZ$ and $g\in 1+p\ZZ_p$. We define the generalized logarithmic function
in $\mathbb Q_p^\times$ by
\begin{equation}\label{genlogqp}
\widetilde{\operatorname{log}}_1(x):=s\operatorname{log}(p)+\omega(\tilde a)\operatorname{log}(g),    
\end{equation}
where $s\operatorname{log}p$ is only formal and represents one disjoint component for each $s$.

Furthermore, the group of characters $\operatorname{Hom}(\mathbb F_p^\times,\ZZ_p^\times)$, a cyclic group of order $p-1$, is generated by Teichm\"{u}ller character $\omega$. Therefore, any character $\chi:F_p^\times\to \ZZ_p^\times$ can be written as $\chi=\omega^l$ for some $l\in \{0,1,\cdots,p-1\}$.
Then we define a family of generalized logarithmic functions
\begin{equation}\label{genlogqpcharacter}
\widetilde{\operatorname{log}}_l(x):=s\operatorname{log}(p)+\omega^l(\tilde a)\operatorname{log}(g).    
\end{equation}

\begin{rem}
The defining equation~\ref{genlogqp} reveals the underlying polar coordinate of $\mathbb Q_p^\times$, which is clearly related to the \'etale picture in Mochizuki's sense. In real field $\RR$, we have $z=r\operatorname{exp}(\sqrt{-1}\cdot\theta)$, and $\operatorname{log}z=\operatorname{log}r+\sqrt{-1}\cdot\theta$ for $\theta\in[0,2\pi]$.   Now the radius $\operatorname{log}r$ is replaced by $s\operatorname{log}(p)$, and the angle $i\theta$ is replaced by $\omega(\tilde a)\operatorname{log}(g)$. Note that $\omega(\tilde a)$ inherits the symmetry from $\mathbb F_p^\times$ and thus is a sort of "rotation".
\end{rem}

\subsection{The differential structure of $\operatorname{Diff}^{\operatorname{an}}(B_\epsilon)$}
From now on throughout this section, we will focus on extending the results of Teichm\"{u}ller theory on  real line discussed in~\cite{Lu2025} to that on a $p$-adic field. First we aim to find an isometry mapping $\operatorname{Diff}^{\operatorname{an}}(B_\epsilon)$
endowed with a norm with respect to a projective structure to a space of functions endowed with a $L^r$-norm with respect to an affine structure. 

We define a variation of $\varphi\in\operatorname{Diff}^{\operatorname{an}}(B_\epsilon)$ as 
\begin{equation}
\varphi^t:\mathbb Q_p\to B_\epsilon \ \ \text{via}:  t\to\varphi(t,x),  
\end{equation}
such that the following conditions are satisfied:

\begin{itemize}
\item[1.] $\varphi^t\in\operatorname{Diff}^{\operatorname{an}}(B_\epsilon)$ for every $t$ in a small neighborhood of $0$ in $\mathbb Q_p$.
\item[2.]
$\varphi^t$ is locally analytic with respect to $t$ at the point $t=0$.
\end{itemize}
Note that the second condition guarantees the local analyticity of the differential along the direction $t$. We should also note that the variable $t$ cannot reside in $\RR$ or $\CC$, for otherwise $\varphi^t$ would never be continuous with respect to $t$. Therefore, the functions used here are significantly different from the ones used in $p$-adic Harmonic analysis, where most of the functions  considered are locally constant test functions and distributions $f:\mathbb Q_p\to \CC$, in which case $t$ can be taken in $\RR$ or $\CC$; see e.g.,~\cite{Albeverio2010}.

A straightforward calculation of variations for $\varphi^t$ around the identity leads to a simple characterization of the Lie algebra.
\begin{prop}
The Lie algebra of $\operatorname{Diff}^{\operatorname{an}}(B_\epsilon)$ is given by
\begin{equation}
\mathfrak g^{\operatorname{an}}=\left\{u: u\in C^{\operatorname{an}}(B_\epsilon)\text{ and }  u(0)=0 \right\}.
\end{equation}    
\end{prop}
\begin{proof}
Denote by $\widetilde D$ (resp. $D$) the differential operator with respect to $t$ (resp. $x\in B_\epsilon$). The compactness of $B_\epsilon$ implies that for an arbitrarily small $\e>0$, there exists a neighborhood $U_0$ of $t=0$ such that uniformly for any $x\in B_\epsilon$, 
\begin{align*}
|f^t-\widetilde D_0\varphi(t)|=|\varphi^t-\varphi^0-\widetilde D_0\varphi(t)|\leq \e|t|,    
\end{align*}
where $f^t\in C^{\operatorname{an}}(B_\epsilon)$, because
$\varphi^t\in\operatorname{Diff}^{\operatorname{an}}(B_\epsilon)$. Consequently, $\widetilde D_0\varphi$ is strictly differentiable at any $x_0$. This comes from the following calculations: for any fixed $t\in U_0$ and an arbitrarily small $0<\e'$, there exists a neighborhood $V_0$ of $x_0$ such that for any $x,y$ in $V_0$,
\begin{align*}
|f^t(x)-f^t(y)-D_{x_0}f^t(x-y)|\leq\e'|x-y|.    
\end{align*}
So we have
\begin{align*}
&|(\widetilde D_0\varphi^t)_{x}-(\widetilde D_0\varphi^t)_{y}-D_{x_0}f^t(x-y)|
\\&\leq\operatorname{max} \left(|(\widetilde D_0\varphi^t)_{x}-f^t(x)|, |(\widetilde D_0\varphi^t)_{y}-f^t(y)|,  |f^t(x)-f^t(y)-D_{x_0}f^t(x-y)|\right)
\\&\leq \operatorname{max} \left(\e'|x-y|, \e|t|\right),
\end{align*}
where $|t|>0$ can be taken arbitrarily small.

In addition, the norm of differential operator $\widetilde D_0\varphi$ vanishes at $x=0$, i.e.,  $\widetilde D_0\varphi(0)=0$ since $f^t(0)=0$. Finally, we show that an arbitrary tangent $u\in C^{\operatorname{an}}(B_\epsilon)$ with $u(0)=0$ would reside in $\mathfrak g^{\operatorname{an}}$. Consider a variation 
$\varphi^t(x)=\operatorname{id}(x)+tu(x)$ for some $t\in \mathbb Q_p$,  its differential with respect to $x$ is given by
$D_{x_0}\varphi=1+tD_{x_0} u$. It is clear that $||tD_{x_0} u||\leq L<1$ if the norm of $t$ is taken to be sufficiently small.
\end{proof}
\begin{rem}
In this proposition, we require the perturbation $f$ in $\operatorname{Diff}^{\operatorname{an}}(B_\epsilon)$ and the variation $u\in g^{\operatorname{an}}$ to be locally analytic. But this requirement for $p$-adic Lie group is superfluous. In fact, we can show that the "Lie algebra" of $\operatorname{Diff}^{1}(B_\epsilon)$
is given by 
\begin{equation}
\mathfrak g^{1}=\left\{u: u\in C^{1}(B_\epsilon)\text{ and }  u(0)=0 \right\}.   
\end{equation}
The situation is similar for "smoothness" in real/complex geometry. Recall that in \cite[Remark~3.1.3]{Lu2025} for real line case, the analysis on Sobolev space $W^{\infty,1}(\RR)$ can be extended to lower regularity, e.g., the space of biLipschitz homeomorphisms $\biLip^{1,1}_{-\infty}(\RR)$.   
\end{rem}
Now we can define the right invariant Finsler metric on $\operatorname{Diff}^{\operatorname{an}}(B_\epsilon)$.
To this end, we write any tangent vector $h\in T_\varphi \operatorname{Diff}^{\operatorname{an}}(B_\epsilon)$ as $X\circ\varphi$ with 
$X\in \mathfrak g^{\operatorname{an}}$.
This allows us to 
define the right-invariant Finsler metric on 
$\operatorname{Diff}^{\operatorname{an}}(B_\epsilon)$ via
\begin{align*}
F_{r,\varphi}(h)&=\left(\int_{B_\epsilon} ||D_x(h\circ\varphi^{-1})||^r |dx|\right)^{1/r}=
 \left(\int_{B_\epsilon} ||D_xX||^r |dx|\right)^{1/r}\\&=\left(\int_{B_\epsilon} ||\frac{D_xh}{D_x\varphi}\circ\varphi^{-1}||^r |dx|\right)^{1/r}=\left(\int_{B_\epsilon} ||D_xh||^r |dx|\right)^{1/r},
\end{align*}
where $|dx|$ is the normalized Haar measure, and the last equality holds by using  the fact $||D_x\varphi||=||1+D_xf||=1$ and change of variable formula~\ref{changevariable}.
\begin{rem}
Compared to the case of real line, the right invariant Finsler metric $F_r$ is much simpler and not of much interest, since the integrand of the defining equation should be some locally constant $p$-adic norm. Consequently, it would be meaningless for us to calculate variations of energy functional defined by $F_r$, e.g., the Euler-Lagrangian.    
\end{rem}
In spite of this limitation in $p$-adic setting,  the differential structure on $\operatorname{Diff}^{\operatorname{an}}(B_\epsilon)$ further allows us to find the isometry between $\operatorname{Diff}^{\operatorname{an}}(B_\epsilon)$
and a function space.
\begin{thm}\label{thm:isometry:m}
For $r\in [1,\infty]$ and $m\in\NN^+$, the mapping
\begin{equation}
\Phi_m: \begin{cases}
\left(\operatorname{Diff}^{\operatorname{an}}(B_\epsilon),  F_r\right)  	&\to  \left(C^{\operatorname{an}}(B_\epsilon),L^r\right)\\ 
\varphi 						&\mapsto \frac{1}{p^m}\left((D_x\varphi)^{p^m}-1\right)
\end{cases}
\end{equation}
is an isometric embedding onto the image $\mathcal U=\Phi_m(\operatorname{Diff}^{\operatorname{an}}(B_\epsilon))$. 
The inverse of $\Phi_m$ is given by
\begin{equation}\label{eq:diffiso:inversem}
\Phi_m^{-1}: \begin{cases}
\mathcal U  &\to  \operatorname{Diff}^{\operatorname{an}}(B_\epsilon) \\ g &\mapsto x+D_x^{-1}\left( \left(p^m g(x)+1\right)^{1/p^m}-1\right),
\end{cases}
\end{equation}
where $D_x^{-1}$ is the inverse of $D_x$ converting the power series to its preimage of differential.
 \end{thm}
\begin{proof}
For any $\varphi=\operatorname{id}+f\in\operatorname{Diff}^{\operatorname{an}}(B_\epsilon)$, we have $f\in C^{\operatorname{an}}(B_\epsilon)$ and $||D_xf||<1$. We calculate the variation formula of the mapping $\Phi_m$
\begin{align*}
D_{\varphi,h}\Phi_m= (D_x\varphi)^{p^m-1} D_xh.    
\end{align*}
Using again the fact $||D_x\varphi||=1$, we have $||D_{\varphi,h}\Phi_m||_{L^r}=||h||_{F_r}$, which proves the isometric embedding of $\Phi_m$. It remains to show that $\Phi_m^{-1}$ is given by formula~\ref{eq:diffiso:inversem} and its image is in $\operatorname{Diff}^{\operatorname{an}}(B_\epsilon)$. For any  $\varphi=\Phi_m^{-1}(g)$, we easily see that $D_x\varphi=\left(p^m g(x)+1\right)^{1/p^m}$, where the right side can be written as a power series. Since the characteristic of a $p$-adic field is zero, $\varphi$ can be uniquely expressed as the form $\varphi=\operatorname{id}+f$, where $f$ is also a power series. But $g\in \mathcal U$ implies that there is already some $\varphi$ mapped to $g$. By uniqueness of such a power series $f$, the $f$ obtained by integration has to satisfy the limitations given by $\varphi\in\operatorname{Diff}^{\operatorname{an}}(B_\epsilon)$. 
\end{proof}

Combining Theorem~\ref{thm:explim} and Corollary~\ref{cor:loglim} with Theorem~\ref{thm:isometry:m},
we obtain a limit of  $\Phi_m$ for $m\to\infty$. 

\begin{thm}\label{thm:isometry:infty}
For $r\in [1,\infty]$ and $m\in\NN^+$, the mapping
\begin{equation}
\Phi_\infty: \begin{cases}
\left(\operatorname{Diff}^{\operatorname{an}}(B_\epsilon),  F_r\right)  	&\to  \left(C^{\operatorname{an}}(B_\epsilon),L^r\right)\\ 
\varphi 						&\mapsto \operatorname{log}(D_x\varphi)
\end{cases}
\end{equation}
is an isometric embedding onto the contractible image $\mathcal U=\Phi_\infty(\operatorname{Diff}^{\operatorname{an}}(B_\epsilon))$.
The inverse of $\Phi_\infty$ is given by
\begin{equation}\label{eq:diffiso:inversem}
\Phi_\infty^{-1}: \begin{cases}
\mathcal U  &\to  \operatorname{Diff}^{\operatorname{an}}(B_\epsilon) \\ g &\mapsto x+D_x^{-1}\left( \operatorname{exp}(g(x))-1\right),
\end{cases}
\end{equation}
where $D_x^{-1}$ is the inverse of $D_x$ converting the power series to its preimage of differential.
 \end{thm}
\begin{proof}
The main difference of the proof here from the former one is that the logarithmic function is a power series now, so we need to consider its convergence. 
For any $\varphi=\operatorname{id}+f\in\operatorname{Diff}^{\operatorname{an}}(B_\epsilon)$, we have $f\in C^{\operatorname{an}}(B_\epsilon)$ and $||D_xf||<1$. But we know that $\operatorname{log}(1+z)$ is convergent for $|z|<1$.
Furthermore, it is easy to see that $\operatorname{log}(1+tD_xf)\to 0$ as $|t|\to0$, so
the image $\mathcal U$
is contractible. Note that $\mathcal U$ may not include all of the locally analytic functions with $||D_xf||<1$. The counter-example is given by $\varphi(x)=x-x^p$, as shown in Remark~\ref{rem:lipmapnec}.

Next we calculate the variation formula of the mapping $\Phi$
\begin{align*}
D_{\varphi,h}\Phi_m=  D_xh/D_x\varphi.    
\end{align*}
Using again the fact $||D_x\varphi||=1$, we have $||D_{\varphi,h}\Phi_m||_{L^r}=||h||_{F_r}$, which proves the isometric embedding of $\Phi_m$. The remaining part of the proof follows almost the same steps in Theorem~\ref{thm:isometry:m}, so we omit it. 
\end{proof}

\begin{rem}
In some literature, the isometric embedding can be expressed as the following logarithmic differential $\Phi_\infty:\varphi\to\operatorname{log}(d\varphi)=\operatorname{log}(\varphi^*dz)$. Note that this is equivalent to our expression. More precisely, we have $\operatorname{log}(d\varphi)=\operatorname{log}(D_z\varphi)+\operatorname{log}(dz)$, which is only a difference generated by translation.      
\end{rem}
One interesting point of view on this limit $\Phi_m\to\Phi_\infty$ is that the multiplicative structure on the slope of diffeomorphism $D_x\varphi$ goes to the additive structure. To make it more clear, the equality
\begin{align*}
\frac{1}{p^m}(D_x(\varphi\circ\phi))^{p^m}=\frac{1}{p^m}(D_x\phi)^{p^m} \cdot (D_x\varphi)^{p^m}\circ\phi    
\end{align*}
after deducting a constant  becomes the following for $p$ goes to infinity
\begin{align*}
\operatorname{log}(D_x(\varphi\circ\phi))=\operatorname{log}(D_x\phi) +\operatorname{log} (D_x\varphi)\circ\phi.    
\end{align*}

\subsection{Schwarzian derivative and Frobenius approximation}\label{subsec:schwarzder}
From now on throughout the section, our analysis involves the differential of order higher than two, so we restrict ourselves to locally analytic functions and ignore the  differentiability or strict differentiability.

As indicated in~\cite[section~3.2]{Lu2025}, the Schwarzian derivative can be viewed as the diagonal elements of the mixed derivative
\begin{equation}
S\{\varphi\}(x):=   6D_yD_zV(y,z)|_{(y,z)=(x,x)}, 
\end{equation}
where $V$ serves as a sort of potential 
\begin{equation}
V(y,z)=\operatorname{log}\left(\frac{\varphi(y)-\varphi(z)}{y-z}\right),    
\end{equation}
for any $\varphi\in\operatorname{Diff}^{\operatorname{an}}(B_\epsilon)$. Then the restriction of $V$ in the diagonal is just the map $\Phi_\infty$ defined in Theorem~\ref{thm:isometry:infty}
\begin{equation}
u_\varphi(x):=\Phi_\infty(\varphi(x))=\operatorname{log}(D_x\varphi)=V(x,x),   
\end{equation}
which satisfies the following equality
\begin{equation}\label{eq:schwazexpression}
S\{\varphi\}(x)=D_x^2u_\varphi(x)-\frac{1}{2}(D_xu_\varphi(x))^2.    
\end{equation}

The limiting process from Theorem~\ref{thm:isometry:m} to Theorem~\ref{thm:isometry:infty} enlightens us of a possible approximation for Schwarzian derivative. Since $\Phi_m$ is the composition of Frobenius morphisms, we call it a Frobenius approximation.
\begin{defi}
The Frobenius potential $V_m$ is defined to be 
\begin{equation}
V_m(y,z):=\frac{1}{p^m}\left\{\left(\frac{\varphi(y)-\varphi(z)}{y-z}\right)^{p^m}-1\right\}.    
\end{equation}
Then the Frobenius Schwarzian derivative is  given by
\begin{equation}
S_{m}\{\varphi\}(x):=   6D_yD_zV_m(y,z)|_{(y,z)=(x,x)}.  
\end{equation}
\end{defi}
Similarly, the restriction of $V_m$ in the diagonal is just the map $\Phi_m$
\begin{equation}
V_m(x,x)=\Phi_m(\varphi).   
\end{equation}
It turns out that the pointwise limit of $S_m$ is just the Schwarzian $S$.
\begin{prop}
 The $L^p$-Schwarzian $S_{p}\{\varphi\}$ has the following expression
 \begin{equation}
 S_{m}\{\varphi\}(x)=\left(\frac{3p^m}{2}(\frac{D_x^2\varphi}{D_x\varphi})^2+S\{\varphi\}(x)\right)(D_x\varphi)^{p^m}.    
 \end{equation}
 Therefore, the limit of $S_m$ is the Schwarzian $S$ when $m\to\infty$.
\end{prop}
\begin{proof}
The calculation is carried out almost the same way as in~\cite[Proposition~3.3.1]{Lu2025}, while the only difference here is that the exponent $1/p$ changes to $p^m$. It remains to show that $S_m\to S$. Note that we have $(D_x\varphi)^{p^m}\to1$ by using formula~\ref{eq:logexpansion}.
\end{proof}

\subsection{The $p$-adic Bers embedding and Teichm\"{u}ller theory on $\operatorname{Diff}^{\operatorname{an}}(B_\epsilon)$}
To study complex structures on compact Riemann surfaces, Bers used the Schwarzian derivative to define a mapping $h=S\{\tilde \varphi\}$ which embeds the universal Teichm\"{u}ller space into an open subset $U$ of the space of bounded holomorphic functions $h$ on $\mathbb H$ (or $\mathbb D$) with $L^\infty$-norm. Henceforth, the Teichm\"{u}ller theory on real line can be viewed as a study on the behavior of Bers embedding on the boundary of upper half plane $\mathbb H$, i.e., a real line; see~\cite{Lu2025}. This interpretation further motivates us to replace the real line by a $p$-adic field, and consider the $p$-adic Bers embedding.

First we introduce the quotient relation of  diffeomorphism groups as shown in~\cite[section~3.4]{Lu2025}. The following group
$\widetilde\Diff^{\operatorname{an}}(\mathbb Q_p)$ is a larger diffeomorphism group that includes all the affine transformations,
\begin{align*}
\widetilde\Diff^{\operatorname{an}}(\mathbb Q_p)=\left\{\varphi=a\cdot\id+b+f: (\varphi-b)/a\in\Diff^{\operatorname{an}}(\mathbb Q_p), a\in\mathbb Q_p^{\times},  b\in \mathbb Q_p \right\}.    
\end{align*}
Then the $p$-adic group
$\Diff^{\operatorname{an}}(\mathbb Q_p)$ is the quotient of $\widetilde\Diff^{\operatorname{an}}(\mathbb Q_p)$ by the affine group of $\mathbb Q_p$, i.e.,
\begin{equation}\label{quotentgroup}
\Diff^{\operatorname{an}}(\mathbb Q_p)=\widetilde\Diff^{\operatorname{an}}(\mathbb Q_p)/\operatorname{A}( \mathbb Q_p),    
\end{equation}
where the real affine group $\operatorname{A}(\mathbb Q_p)=\{\varphi: \varphi(x)=ax+b,\; a\in\mathbb Q_p^{\times},  b\in \mathbb Q_p\}$ acts by right cosets.
Indeed, we have the following identity
\begin{align*}
a\cdot\id+b+f=(a\cdot\id+b)\circ(\id+f/a).    
\end{align*}
In addition, the space of all vertical shifts of $C^{\operatorname{an}}(\mathbb Q_p)$ is given by
\begin{equation}
C_+^{\operatorname{an}}(\mathbb Q_p)=  \left\{g+r\operatorname{log}p:g\in C^{\operatorname{an}}(\mathbb Q_p), \ r\in \ZZ\right\}\cong\bigsqcup_{r\in\ZZ}C^{\operatorname{an}}(\mathbb Q_p),  \end{equation}note that the word "vertical shift" comes from the real line case.
Then $\Phi_\infty$ can be extended to $\widetilde\Diff^{\operatorname{an}}(\mathbb Q_p)$, and is no longer an injective isometry, but rather a covering map. 
\begin{prop}\label{affinegroupiso}
The mapping
\begin{equation}
\widetilde\Phi_\infty: \begin{cases}
\left(\widetilde\Diff^{\operatorname{an}}(\mathbb Q_p), F_r\right)  	&\to  \left(C_+^{\operatorname{an}}(\mathbb Q_p),L^r\right)\\ 
\varphi 						&\mapsto \widetilde{\operatorname{log}}_l(D_x\varphi)
\end{cases}
\end{equation}
is a covering map onto its image with the fibre isomorphic to $\mathbb F_p^\times\times\mathbb Q_p$. %Furthermore, $\widetilde{\operatorname{\Phi}}_\infty$ is the projection of a trivial bundle.   
\end{prop}

\begin{proof}
For any 
$\varphi=a\operatorname{id}+b+f$ in
$\widetilde\Diff^{\operatorname{an}}(\mathbb Q_p)$,
we have $||D_xf||<||a||$. Consequently, the order of $D_x\varphi$ is exactly the order of $a$.
The term $s\operatorname{log}p$ appearing in $\widetilde{\operatorname{log}}_l(D_x\varphi)$ with respect to decomposition~\ref{genlogqp} remains invariant when $x$ goes through $\mathbb Q_p$, so the image of  $\widetilde\Phi_\infty$ is indeed in $C_+^{\operatorname{an}}(\mathbb Q_p)$.

It is left to characterize the fibre of $\widetilde\Phi_\infty$. 
Assume that $\varphi_i=a_i\operatorname{id}+b_i+f_i$ and $D_x\varphi_i=p^{s_i}u_ig_i$ such that $s_i\in\ZZ$ and $g_i\in 1+p\ZZ_p$.
Then the following equivalence relation holds
\begin{align*}
\widetilde{\operatorname{log}}_l(D_x\varphi_1)=\widetilde{\operatorname{log}}_l(D_x\varphi_2)
\Longleftrightarrow
s_1=s_2 \;\text{and} \; \omega^l(\tilde u_1 \tilde u_2^{-1})\operatorname{log}(g_1)=\operatorname{log}(g_2),
\end{align*}
where $\tilde u_i\in \mathbb F_p^\times$ is the residue class of  $u_i \operatorname{mod}p$. Thus the fibre of each $\widetilde{\operatorname{log}}_l(D_x\varphi)\in C^{\operatorname{an}}(\mathbb Q_p)$ is an arbitrary point $(u,b)\in \mathbb F_p^\times\times\mathbb Q_p$. 
\end{proof}
Now we can present the $p$-adic Bers embedding on the quotient space $\Diff^{\operatorname{an}}(\mathbb Q_p)$. Maybe the term "embedding" is  slightly inappropriate, since tangent map $T\beta$ is only injective in a closed-open neighborhood of zero.
\begin{thm}
The mapping 
\begin{equation}
\beta: \begin{cases}
\left(\Diff^{\operatorname{an}}(\mathbb Q_p), F_r\right)  	&\to  \left(C^{\operatorname{an}}(\mathbb Q_p),L^r\right)\\ 
\varphi 						&\mapsto S\{\varphi\}
\end{cases}
\end{equation} 
is an injective mapping onto a subset of 
 $C^{\operatorname{an}}(B)$, and the tangent map $D\beta$ is an injective in the unit open ball $\{v:|v|<1\}$ around zero.
Here $S$ is the Schwarzian derivative. 
\end{thm}
\begin{proof}
First we show that $\beta$ is indeed an embedding. Let $\varphi=\operatorname{id}+f$ and $h=\operatorname{log}(D_x\varphi)$. Using formula~\ref{eq:schwazexpression}, the mapping $\beta$ and its differential can be expressed
as the following
\begin{align*}
\beta:\varphi\to D_x^2h-\frac{1}{2}(D_xh)^2, \; \;  D\beta:\delta\varphi\to D_x^2(\delta h)-D_xh \cdot D_x(\delta h),    
\end{align*}
where $\delta h$ is a variation of $h$.  Hence we read off the kernel of $D\beta$ 
\begin{equation}\label{ode1}
\delta h=x+D_x^{-1}(\operatorname{exp}(h)-1)
\end{equation} 
But using Taylor expansion~\ref{logtaylorexp} and $||D_x f||<1$, we see that $||h||<1$ and using assumption $||\delta h||<1$, there is a contradiction if $\delta h\neq0$. The map $\beta$ is  injective
simply by solving the differential equation (well-known in the complex case): the only possible transformations that annul the Schwarzian are the Mobius transformations, which are already eliminated by quotient relation~\ref{quotentgroup}.  
\end{proof}
\begin{rem}
In the paper~\cite{Lu2025}, the image of real Bers embedding is proved to be unbounded and non-open, it remains a question how  to characterize the $p$-adic Bers embedding by its image.    
\end{rem}

\subsection{$p$-adic Fisher-Rao metric and location-scale family}
In preceding subsection, we define a right-invariant Finsler metric $F_r$ which does not carry a good structure for analysis. For this reason, we introduce another type of metric, i.e., the Fisher-Rao metric.  The classical Fisher-Rao information metric $g$ 
\begin{equation}\label{fisherraometric}
g_{ij}(\theta)=\int_X \frac{\partial \operatorname{log}\rho(x;\theta)}{\partial\theta_i}\frac{\partial \operatorname{log}\rho(x;\theta)}{\partial\theta_j}\rho(x;\theta)dx.     
\end{equation}
is defined for finite dimensional statistical models, and play a central role in information geometry; see e.g.,~\cite{Amari2000}. More recent efforts focus on considering the Fisher-Rao metric as a unparametrized one on the space of probability densities of
a compact manifold $M$, and in his paper~\cite{khesin2011}, Khesin et. al. showed us that one can interpret Fisher-Rao metric as a sort of right-invariant, homogeneous Sobolev metric on a group of diffeomorphisms. We use this interpretation to define a Fisher-Rao type metric on $\Diff^{\operatorname{an}}(B_\epsilon)$, and adopt a more generalized version, i.e., the $L^r$-Fisher-Rao metric, which was introduced in~\cite{Lu2023,bauer2023}.
\begin{defi}
Let $r\in[1,\infty]$.
For any $\varphi=\operatorname{id}+f\in\Diff^{\operatorname{an}}(B_\epsilon)$ with $D_xf$ non-vanishing, and any tangent vector $h\in T_\varphi \Diff^{\operatorname{an}}(B_\epsilon)$, the $L^r$-Fisher-Rao metric is defined as 
\begin{equation}\label{FRmetric}
||h||_{\operatorname{FR}}:= \left(\int_{B_\epsilon}||D_xh||^r||{D_x f}||^{1-r} |dx|\right)^{1/r}.  
\end{equation}
\end{defi}
We should note that the classical Fisher-Rao metric is equivalent to our definition of $L^2$-Fisher-Rao metric on $\Diff^{\operatorname{an}}(B_\epsilon)$. That is to say, the differential $D_xf$ of perturbation function $f$  can be seen as a sort of probability density function  $p$ in formula~\ref{fisherraometric}.
\begin{rem}
Since the ball $B_\epsilon$ is compact for any finite $\epsilon>0$,  there exist  positive numbers $L$ and $M$ such that the locally analytic functions $h$ and $f$ satisfy $||D_xh||\leq M$ and $||D_xf||\geq L$. Therefore, the integral in the definition is convergent. For some reason, we are interested in the case for which $B_\epsilon$ is not compact, i.e., $B_\infty=\mathbb Q_p$. 
In this case, not all the functions $f\in C^{\operatorname{an}}(B_\epsilon)$ generate a convergent integral. 
\end{rem}
Next we consider a perturbation function $f$ in the set 
\begin{align*}
P:=\{f: \varphi=\operatorname{id}+f\in\Diff^{\operatorname{an}}(\mathbb Q_p), \; D_xf\neq0\}.    
\end{align*}
The location-scale family of $f$ in $\mathbb Q_p$ is defined as 
\begin{equation}
\rho(x;t,\sigma):=\frac{1}{\sigma}D_xf(\frac{x-t}{\sigma}),    
\end{equation}
where $x,t\in \mathbb Q_p$ and $\sigma\in\mathbb Q_p^\times$. Let $z=\frac{x-t}{\sigma}$, then we have an equality of differential forms, 
\begin{align*}
\rho(x;t,\sigma)dx=D_zf(z)dz.   
\end{align*}

For the real line case indicated in~\cite{Lu2025},    we see that each of these families  endowed with the Fisher-Rao information metric induces a hyperbolic structure on the upper half plane $\mathbb H$. For the $p$-adic case, there is no such upper half plane since the $p$-adic field is totally disconnected and there is no ordered structure such that the imaginary part $>0$ makes sense. Consequently, we  define the domain that carries the hyperbolic structure to be  
\begin{align*}
A_\infty:=\{(t,\sigma): t\in \mathbb Q_p, \sigma\in \mathbb Q_p^\times\}  \; \; \text{on} \; \Diff^{\operatorname{an}}(\mathbb Q_p).   
\end{align*}
This corresponds to 
the union of the upper half plane $\mathbb H$ and its dual, the lower half plane  $\mathbb H^*$ in the real line case, where we only need to take one piece of the two disjoint half planes since $\mathbb R^\times=\RR^+\cup\RR^-$. 
Now we can state the $p$-adic version of the correspondence between Fisher-Rao metric and hyperbolic structure. 
\begin{thm}\label{thm:fisherrao}
Let the location-scale family $\rho(x;t,\sigma)$  of a perturbation function $f\in P$ be convergent for $p$-adic Fisher-Rao metric~\ref{FRmetric}, then the metric on the location-scale family  is a  hyperbolic metric on  $A_\infty$ spanned by $\sigma$ and $t$. 
\end{thm}
\begin{proof}
Set $h=D_xf$ for a fixed $f\in P$. Then the logarithmic probability density is given by
\begin{align*}
\operatorname{log}\rho=-\operatorname{log}\sigma+  \operatorname{log}h(z).  
\end{align*}
The partial derivatives can be computed as 
\begin{align*}
\partial_t\operatorname{log}\rho=-\frac{D_zh(z)}{\sigma h(z)}, \; \;  \partial_\sigma\operatorname{log}\rho=-\frac{1}{\sigma }\left(1+z\frac{D_zh(z)}{h(z)}\right).   
\end{align*}
Next we compute each component of Fisher-Rao metric~\ref{fisherraometric}
\begin{align*}
g_{tt}&=\frac{1}{|\sigma|^2}\int_{B_\epsilon} \left|\frac{D_zh(z)}{h(z)}\right|^2 |h(z)dz|, \; \; g_{\sigma\sigma}=\frac{1}{|\sigma|^2}\int_{B_\epsilon} \left|1+z\frac{D_zh(z)}{h(z)}\right|^2 |h(z)dz|,
\\g_{t\sigma}&=\frac{1}{|\sigma|^2}\int_{B_\epsilon} \left|\frac{D_zh(z)}{h(z)}\left(1+z\frac{D_zh(z)}{h(z)}\right)\right| |h(z)dz|.  
\end{align*}
Then the infinitesimal  metric $g$ on an open subset is written as
\begin{align*}
&\int_{B_\epsilon} \left|\left(\frac{D_zh(z)}{h(z)}\right)^2dt^2+\left(1+z\frac{D_zh(z)}{h(z)}\right)d\sigma^2
+ 2\frac{D_zh(z)}{h(z)}\left(1+z\frac{D_zh(z)}{h(z)}\right)dtd\sigma\right| \frac{|h(z)dz|}{|\sigma|^2}
\\&\leq\operatorname{max}\{g_{tt}|dt|^2, g_{\sigma\sigma}|d\sigma|^2, g_{t\sigma}|dt||d\sigma|\}.
\end{align*}
This non-archimedean norm gives a hyperbolic structure on $\mathbb Q_p\times \mathbb Q_p^{\times}$. 
\end{proof}
It is clear that the integrand of $g$ in the theorem is non-vanishing for any $z$ in $B_\epsilon$. Indeed, if $D_zh(z)=0$, then the second term $1+zD_zh(z)/h(z)$ equals one. But in some direction with respect to the slope $dt/d\sigma$, 
the integral may vanish. If this does not happen, we have the following corollary by applying Serre's theorem~\ref{serrethm2}. 
\begin{cor}
Let $F(z):=-(zD_zh(z)+h(z))/D_zh(z)$. If the slope $dt/d\sigma$ has no intersection with the image of $F$, then the integrand of $g$ is nonwhere vanishing and  
$g$ is
equal to  the Serre invariant $i(B_\epsilon)$ modulo $p-1$. 
\end{cor}

\section{$p$-adic Teichm\"{u}ller theory}\label{sec:padicteichmuller}
The $p$-adic Teichm\"{u}ller theory developed by  Mochizuki is a generalization of the Fuchsian uniformization to the $p$-adic context; see e.g.,~\cite{mochizuki96, mochizuki99, mochizuki02}. It mainly focuses on the hyperbolic curve $C$ over a positive characteristic perfect field $k$, with nilpotent ordinary indigenous bundle $P$ over $C$. Mochizuki"s algebraic approach retrieves information of certain classes of hyperbolic curves from its \'etale fundamental group, now known as anabelian geometry. In this section, we will introduce the topic of $p$-adic Teichm\"{u}ller theory from an analytic point of view, which is based on the relationship between $p$-adic Lie group and $p$-adic Teichm\"{u}ller space discussed in the preceding section.

\subsection{Teichm\"{u}ller characters and Witt vectors}
Let $\mathcal O_k$ be a complete discrete valuation ring with the residue field $k$, which is a perfect field with characteristic $p>0$. In  the following, we consider a simple case, the $p$-adic integers $\ZZ_p$ with the residue field $\ZZ_p/p\ZZ_p\cong\mathbb F_p$, i.e., the finite field of $p$ elements.

The Teichm\"{u}ller character $\omega$ is a character of $\mathbb F_p^\times$ (resp. $\mathbb F_4^\times$), when $p$ is an odd prime
(resp. $p=2$), taking values in the root of units in $p$-adic integers. It chooses for each residue class $x$ (mod $p$) a unique canonical representative such that 
$\omega(x)^{p-1}=1$. The existence and uniqueness of such a representative is a consequence of Hensel's lemma. In addition, we indeed have a multiplicative character by applying the uniqueness to the fact that $\omega(x)\omega(y)$ is again a $p-1$-th root and $\omega(xy)\equiv\omega(x)\omega(y)$ (mod $p$). 

This character $\omega:\mathbb F_p^\times\to \ZZ_p^\times $ can be viewed as a limit of compositions of Frobenius morphisms, similar to the situation shown in~\ref{eq:limlog}. The following proposition is in~\cite[Exercise~I.5.19]{koblitz1977}, for which we give a proof since our arguments heavily depend on this proposition.
\begin{prop}\label{teichmullerchar}
Let  $x\in\ZZ_p$ be an element in the residue class $\tilde x\equiv\tilde a\ (\operatorname{mod} p)$ for $a\in\{1,2,\cdots,p-1\}$.  Then $x^{p^m}$ converges to  $\omega(\tilde a)$ as $m\to\infty$. 
\end{prop}
\begin{proof}
Firstly we take $x\in\{1,2,\cdots,p-1\}$. By Euler's theorem, we have  \begin{equation}
x^{p^m}\equiv x^{p^{m-1}} (\operatorname{mod} p^m),   
\end{equation}  
which exactly means convergence in $p$-adic sense. We denote by $x_0$ its converging point. By passing to the limit, we obtain $x_0^p=x_0$. Then we have $\omega(\tilde x)=x_0$ by applying the uniqueness argument to the case $x_0\equiv x^{p^m}\equiv x (\operatorname{mod} p)$.

It remains to show if $x=a+pr$ for $r\in\ZZ_p$ and $a\in\{1,2,\cdots,p-1\}$, then $x^{p^m}$ converges to  $\omega(\tilde a)$. But from the Taylor series of $x^{p^m}$, we see that
\begin{equation}
x^{p^m}=(a+pr)^{p^m}=a^{p^m}+p^{m+1}r a^{p^m-1}+\cdots\equiv a^{p^m} (\operatorname{mod} p^m),  
\end{equation}
which implies $\omega(\tilde x)=\omega(\tilde a)$.
\end{proof}

Note that the character holds in a more general setting, that is to say, there exists a unique multiplicative section $\omega:k\to\mathcal O_K$ 
for the natural surjection
$\pi:\mathcal O_K\to k$, 
here $K$ is the $p$-adic completion of $k$.
\begin{rem}
The multiplicative group of $p$-adic units is given by 
\begin{equation}\label{decomppadic}
\ZZ_p^\times\cong U\times(1+p\ZZ_p),    
\end{equation}
where $U$ is the finite cyclic group of roots of unity of order $p-1$ (resp. $2$), when $p$ is an odd prime (resp. $p=2$). More precisely, the character  
$\omega:\mathbb F_p^\times\to  U$ gives a canonical isomorphism between these two multiplicative groups. The decomposition~\ref{decomppadic} is in closed relationship with the quotient relation~\ref{quotentgroup}: both are broken up into two different parts, one is a perturbation sufficiently near identity, the other is an affine transformation
(resp. a cyclic "rotation"). This can be further compared with the case of a compact Riemann surface, the group of diffeomorphisms or quasiconformal mappings can be seen as a composition of two distinct parts, one is the set of M\"{o}bius transformations or the Fuchsian group for its discete subgroups, the other is a perturbation of identity that gives rise to the local continuous properties. This interpretation has even more profound implications, the inter-universe Teichm\"{u}ller theory~\cite{mochizuki12} is divided, in Mochizuki's sense, into two different patterns: one is called the \'etale picture, which is more concerned with global transformations (discrete in some sense) that can be characterized by its \'etale fundamental group; the other is called the Frobenius picture, which is more concerned with local deformations (continuous/differentiable in some sense) that can be characterized by its analytic properties or differential geometry in Teichm\"{u}ller space. The latter one gets its name "Frobenius" from the fact that the log-category is a limit of Frobenius morphisms, which we have observed in Theorem~\ref{thm:isometry:infty}.
\end{rem}

The Teichm\"{u}ller character is remarkable for providing a canonical lifting from the residue field to the $p$-adic field, for this reason it is also known as Teichm\"{u}ller
lift. It will play an important role in our treatment of $p$-adic Teichm\"{u}ller since now we can embed the finite perfect field into an $p$-adic field, and adopt the methods or techniques of a $p$-adic field. 

One way to fulfill this purpose is to introduce the Witt vector, an infinite
sequence of elements of a commutative ring $R$. 
\begin{defi}
A Witt vector over a commutative ring $R$ is a sequence $(X_0, X_1,\cdots)$ of elements of $R$. The Witt polynomials $W_i$ can be defined by $W_0=X_0$, $W_1=X_0^p+pX_1$, and in general
\begin{equation}
W_n=\sum_{i=0}^{n}p^iX_i^{p^{n-i}}.     
\end{equation} 
The $W_n$, also denoted by $X^{(n)}$, are called the ghost components of the Witt vector. 

The ring of Witt vectors $W(R)$ is defined by a triple $(X^{i},+,\cdot)$ with componentwise addition and multiplication of the ghost components, i.e., 
$X^{(i)}+Y^{(i)}=(X+Y)^{(i)}$ and $X^{(i)}Y^{(i)}=(XY)^{(i)}$.
\end{defi}
In particular, for $R=\mathbb F_p$, we have $W(\mathbb F_p)\cong\ZZ_p$. This can be done by representing each $a\in\ZZ_p$ as 
\begin{align*}
a=\omega(a_0)+\omega(a_1)p+\omega(a_2)p^2+\cdots,    
\end{align*}
where $\omega:\mathbb F_p\to\ZZ_p$ is the Teichm\"{u}ller character with $\omega(0)=0$. Then the corresponding Witt vector is given by $(a_0,a_1,\dots,)\in W(\mathbb F_p)$.

Two basic operators on Witt vectors are the Verschiebung operator $V$ and Frobenius operator $F$.
\begin{defi}
The Verschiebung operator $V$ is a shift operator in $W(R)$ taking $(a_0, a_1,\cdots)$ to $(0, a_0, a_1,\cdots)$. The Frobenius operator $F$ is a Frobenius morphism in each component of $W(R)$, taking $(a_0, a_1,\cdots)$ to $(a_0^p, a_1^p,\cdots)$.    
\end{defi}
The two operators are dual to each other, i.e., $VF=FV=[p]$, where $[p]$ is the multiplication by $p$. Particularly, when $R=\mathbb F_p$, $V=[p]$ and $F=\operatorname{id}$, the duality is trivially true.

Let $k$ be a perfect field of characteristic $p>0$. The ring of Witt vectors
$W(k)$
carries a natural non-archimedean norm that makes it into a complete discrete valuation ring. 
More precisely, there is a uniformizer $p\in W(k)$
such that every non-zero ideal is of the form $(p^n)$, the residue field is $W(k)/pW(k)\cong k$, and the field of fractions of $W(k)$ is a complete discrete valuation field. Then following the same procedures to construct a non-archimedean norm on $\ZZ_p$, we obtain a non-archimedean norm on the fraction field of $W(k)$. 
Finally, it should be noted that almost all the preceding arguments regarding the Lie group structure on a $p$-adic field can be extended smoothly to those on a 
fraction field of $W(k)$.
%For this reason, we can reduce to the case $k=\mathbb F_p$ without much loss of generality.

\subsection{Automorphisms of a principal bundle and Indigenous bundles}
The notion of an indigenous bundle is central to Mochizuki's treatment of $p$-adic Teichm\"{u}ller theory. It was originally introduced by Gunning in~\cite{gunning1967} as a flat fibre bundle associated to the complex projective or affine structure. Mochizuki's approach mainly focuses on its algebraic aspect, and considers its relation with the natural monodromy representation of the fundamental group of a hyperbolic curve. However, this interpretation is not quite compatible with our analytic approach to a $p$-adic Lie group. Therefore, we will adopt a new diff-geometric point of view that is dramatically different from the algebraic one. First we recall the notion of a principal bundle, and for our use, we also generalize it to a $p$-adic field.
\begin{defi}
A principal bundle $(P,\pi, M,G)$ is a fiber bundle $\pi:P\to M$ consisting of a total space $P$, a based space $M$ and a structure
group $G$ such that $G$ preserves the fibres $P_x$ of $P$ and acts freely and transitively, i.e., for each $x\in X$ and $ y\in P_x$, the map $G\to P_x$ via $g\to yg$ is a isomorphism.
\end{defi}
Let $\operatorname{Aut}(P)$ be the group of automorphisms of $P$, i.e., bundle maps $\Psi:P\to P$ covering the diffeomorphism $\psi\in\operatorname{Diff}_{\operatorname{id}}(M)$ and are isomorphisms on fibres. In their paper~\cite{abbati1989}, Abbati et. al. showed that $\operatorname{Aut}(P)$ is the total space of a $C^\infty_c$-principal bundle $(\operatorname{Aut}(P),\tilde\pi, \operatorname{Diff}_{\operatorname{id}}(M),\operatorname{Gau}(P))$, where $\operatorname{Diff}_{\operatorname{id}}(M)$ is an open subgroup of $M$ containing the connected component of the identity, and $\operatorname{Gau}(P)$ denotes gauge group of $P$. This gives rise to a short exact sequence
\begin{equation}
1\to \operatorname{Gau}(P)\to \operatorname{Aut}(P)\to \operatorname{Diff}_{\operatorname{id}}(M)\to1.     
\end{equation}
In essence, their approach is based on the identification of $\operatorname{Hom}(P, P')$ with the set of 
smooth sections of a bundle, which can be seen as a 
manifold modelled on a suitable nuclear space, i.e., an inductive limit of nuclear Fr\'echet spaces.%(a $NLF$-space)

On the other hand,  smoothness is replaced by local analyticity for an infinite dimensional Lie group modelled on a $p$-adic field, and it is well-known that the space of locally analytic functions on a compact open $p$-adic manifold is an inductive limit of nuclear spaces. So there should be no difficulty generalizing the principal bundle structure on $\operatorname{Aut}(P)$ to the $p$-adic case, then the short exact sequence becomes 
\begin{equation}\label{padicexactseq}
1\to \operatorname{Gau}(P)\to \operatorname{Aut}(P)\to \operatorname{Diff}^{\operatorname{an}}(B_\epsilon)\to1,    
\end{equation}
where $\operatorname{Aut}(P)$ covers the $p$-adic Lie group $\operatorname{Diff}^{\operatorname{an}}(B_\epsilon)$.

Now we restrict ourselves to a specific case of principal bundles on a Riemann surface or a $p$-adic curve, an indigenous bundle.
The bundle indigenous to a Riemann surface $M$ is a complex projective  (resp. affine) line bundle $P$ such that there exists a complex projective (resp. affine) chart $(U_\alpha,z_\alpha)$, the associated transition functions $g_{\alpha\beta}$ are constant M\"{o}bius (resp. affine) transformations.  
Note that the indigenous bundle $P$ is a flat line bundle with fibre isomorphic to $  P_\CC^1$ (resp. $\CC$). Hence 
$\operatorname{Aut}(P)$ is decomposed as in sequence~\ref{padicexactseq} since an affine indigenous bundle $P$ is a principal bundle, while this is not applicable to the projective case since $  P_\CC^1$ is not a group.

For $p$-adic case, it manifests in 
exact sequence~\ref{padicexactseq} that $\operatorname{Aut}(P)$
of an affine indigenous bundle $P$ can be decomposed into two distinct parts, one is the Lie group $\operatorname{Diff}^{\operatorname{an}}(B_\epsilon)$ pertaining to perturbations of the identity, the other is the gauge group  $\operatorname{Gau}(P)$ on $B_\epsilon$ with each fibre seen as an affine transformation.
\begin{prop}
There is a family of locally constant sections in $\operatorname{Gau}(P)$, which is locally isomorphic to the group of affine transformations.     
\end{prop}
\begin{proof}
Since $P$ is an indigenous bundle, the transition functions are constant affine transformations. We can assign an arbitrary constant affine transformation $\varphi_\alpha$ to a fixed local coordinate $U_\alpha$. This determines a globally defined section by applying each constant transition function to this transformation in each intersection of coordinates. Note that the initial choice of $\varphi_\alpha$ ranges over the entire group of affine transformations. 
\end{proof}
Consequently, the semi-direct product  $A(\mathbb Q_p)\ltimes\operatorname{Diff}^{\operatorname{an}}(B_\epsilon)|_{U_\alpha}$  is the local form of a locally constant element in $\operatorname{Aut}(P)$. In particular for $B_\infty=\mathbb Q_p$, we have seen in 
Corollary~\ref{diffgrouplarge} that this semi-direct product can be realized  as a $p$-adic Lie group $\widetilde{\operatorname{Diff}}^{\operatorname{an}}(\mathbb Q_p)$. So what we did in the last section actually characterizes the local diff-geometric aspect of $\operatorname{Aut}(P)$, for an indigenous affine line bundle $P$. 

\subsection{Projective indigenous bundle}
In the paper~\cite{gunning1967}, Gunning showed that on a  Riemann surface, a complex affine structure is associated to an elliptic curve, while a complex projective structure is associated to a surface of genus $g>1$, i.e., a hyperbolic curve. This has already been implied in the classification theorem of
one dimensional connected group variety. If an algebraic curve $C$ has an affine indigenous bundle, then the affine transformation $f(z)=az+b$ is composed of two kinds of group actions: the rescaling scalar $a$ is in the multiplicative group $K^\times$; the translation $b$ is in the additive group $K$. Hence an affine indigenous bundle on $C$ in essence 
corresponds to a group structure on $C$,  including all non-hyperbolic cases. If $C$ is hyperbolic, a theorem of Hurwitz says that $C$ has 
at most $84(g-1)$ automorphisms, then $C$ cannot be a group variety since a group structure on $C$ generates infinitely many automorphisms.

Therefore, there are two types of $\operatorname{Gau}(P)$ in sequence~\ref{padicexactseq}. One is the case that $\operatorname{Gau}(P)$ is an affine group, which gives rise to the case of multiplicative group and additive group.
The other is the case that  $\operatorname{Gau}(P)$ is essentially an elliptic curve, which will be discussed in the next section.

In the following we consider the projective structure on 
a $p$-adic hyperbolic curve $C$ over $p$-adic field $K$ of genus $g>1$. It is well-known that the automorphism group of projective line $\mathbb P_K^1$ is the projective linear group 
$\operatorname{PGL}_2(K)$.
So we can relate a  
$p$-adic
projective indigenous bundle to a projective structure following the same steps as shown in~\cite[Theorem 2]{gunning1967}.  
\begin{prop}\label{prop:projective}
Let $K$ be a $p$-adic field containing all square roots of $\mathbb Z_p^\times$. If a flat $p$-adic projective bundle $\varphi^0$ is indigenous to $C/K$, then 
it is associated to a flat vector bundle $\varphi$ with respect to a unique projective structure.%flat bundle $\varphi$ such that $\operatorname{det}(\varphi)=1$and $\operatorname{deg}(\operatorname{div}(\varphi))=g-1$ 
\end{prop}
\begin{proof}
The key observation here is that the transition function of indigenous bundle defines the transition function of canonical bundle $K_C$ of $C$. 
More precisely, given a projective coordinate covering $\{U_{\alpha},z_\alpha\}$ of $C$, the coordinate change and its derivative in intersection $U_\alpha\cap U_\beta$ are given by 
\begin{equation}\label{mobiustran}
z_\alpha=\varphi^0_{\alpha\beta}(z_\beta)=\frac{a_{\alpha\beta}z_\beta+b_{\alpha\beta}}{c_{\alpha\beta}z_\beta+d_{\alpha\beta}}, \; \; \;
\frac{dz_\alpha}{dz_\beta}=\frac{1}{(c_{\alpha\beta}z_\beta+d_{\alpha\beta})^2},
\end{equation}
where the determinant of transition matrix is $1$ after possible normalization. Therefore, a holomorphic $1$-form transforms in the form 
$f_\alpha dz_\alpha=f_\beta dz_\beta$, which implies that $f_\alpha=(c_{\alpha\beta}z_\beta+d_{\alpha\beta})^2 f_\beta$.
%The present proof only differs from the one in~\cite{gunning1967}
Since $\operatorname{deg}(K_C)=2g-2$, there exists a line bundle $\xi$ of degree $0$ such that the bundle $K_C\otimes\xi$ has cross-sections $g_\alpha$ with the divisor consisting of $g-1$ double zeros. The transition functions  $\xi_{\alpha\beta}$ of $\xi$ can be normalized to $|\xi|=1$, since the degree of $\xi$ is $0$. Therefore, the transition functions $g_\alpha$ satisfy
\begin{equation}
g_\alpha=\xi_{\alpha\beta}(c_{\alpha\beta}z_\beta+d_{\alpha\beta})^2 g_\beta. 
\end{equation}
Then in each neighborhood, we select a branch of functions $h_{2\alpha}=g_\alpha^{1/2}$ satisfying
\begin{equation}
h_{2\alpha}=\chi_{\alpha\beta}(c_{\alpha\beta}z_\beta+d_{\alpha\beta}) h_{2\beta},    
\end{equation}
where $\chi_{\alpha\beta}$ is the square root of $\xi_{\alpha\beta}$, whose existence is guaranteed by assumption. The pair of functions $h_{1\alpha}:=z_\alpha h_{2\alpha}$ and $h_{2\alpha}$ is the desired cross-section to construct a flat vector bundle $\varphi$. The remaining part of the proof proceeds exactly the same as in~\cite{gunning1967}, so we omit it.  
\end{proof}
Let $\widehat C$ be a hyperbolic curve over a positive characteristic perfect field $k$, and let $\widehat P$ be a projective indigenous bundle over $\widehat C$. Using the ring of Witt vectors, we can complete $k$ to a $p$-adic field $K$. Henceforth we have a projective indigenous bundle $P$ over $C$ with the fibre being a projective line $\mathbb P_K^1$
\begin{equation}
1\to\mathbb P_K^1\to P\to C\to1.
\end{equation}
Since the projective line $\mathbb P_K^1$ is by no means a group, we can no longer obtain an exact sequence as in sequence~\ref{padicexactseq} using the method employed by~Abbati et. al.~\cite{abbati1989}.  
 But if we only consider its 
 local section on a $p$-adic ball $B_\epsilon$,
\begin{equation}
1\to\mathbb P_K^1\to P|_{B_\epsilon}\to B_\epsilon\to1.
\end{equation} 
Then we have the following exact sequence for projective indigenous bundle $P$
\begin{equation}
1\to\operatorname{Aut}(\mathbb P_K^1)\to\operatorname{Aut}(P|_{B_\epsilon})\to \operatorname{Diff}^{\operatorname{an}}(B_\epsilon)\to1.   
\end{equation}

%Note that if $C$ is a compact curve, then Theorem~\ref{serrethm1} claims that $C$ is isomorphic a copy of $p$-adic balls, so the notion $\operatorname{Diff}^{\operatorname{an}}(C)$ makes sense for $C=B_\epsilon$; the exact number of copies of balls can be calculated by the Serre invariant in Theorem~\ref{serrethm2}.
Henceforth, we have two pictures of $\operatorname{Aut}(P|_{B_\epsilon})$: the \'etale-like picture corresponds to the automorphism group $\operatorname{Aut}(\mathbb P_K^1)$, the set of all M\"{o}bius transformations in $K$; while the Frobenius-like picture corresponds to $\operatorname{Diff}^{\operatorname{an}}(B_\epsilon)$, the group of diffeomorphisms sufficiently close to the identity. The \'etale-like picture of $\operatorname{Aut}(P)$ should be the quotient of $\operatorname{PGL}_2(K)$ by a Fushsian subgroup $\Gamma$, since the automorphism group of $C$ is finite. 
We are especially interested in the Frobenius-like picture as it gives us an analytic structure due to completion of $k$. So in the following subsection, we present how canonical Frobenius lifting yields the desired local analytic structure from positive characteristic fields. 

\subsection{Frobenius lifting and canonical lifting}
As illustrated by the diagram in~\cite[III Figure~I.2]{mochizuki12}, there are two kinds of liftings in $p$-adic Teichm\"{u}ller theory. The Frobenius lifting $F$ is the one coming from the Frobenius morphism in positive characteristic, and it gives rise to a log-structure (log-link in IUT). 

For the simple case $k=\mathbb F_p$, the Frobenius lifting from $\mathbb Z/p$ to $\mathbb Z/p^2$ is of the form
\begin{equation}
F: \mathbb Z/p\to \mathbb Z/p^2 \; \; \text{via}:    t\to (1+t)^p-1.
\end{equation}
It is clear that $F$ is well-defined, and $F(t)\equiv t \; \operatorname{mod} p$. This can be extended to a lifting from $\mathbb Z/p$ to $\mathbb Z/p^m$
\begin{equation}
F_m: \mathbb Z/p\to \mathbb Z/p^{m+1} \; \; \text{via}:    t\to (1+t)^{p^m}-1.
\end{equation}
$F_m$ is also well-defined, and 
\begin{equation}
F_m(t)\equiv F_{m-1}(t)\; \operatorname{mod} p^m,    
\end{equation}
 as a result of $(1+t)^{p^{m}}\equiv (1+t)^{p^{m-1}} \; \operatorname{mod} p^m$. 
So it is not hard to see that the limit of $F_m(t)$ in $\ZZ_p$ for $m\to\infty$ is just the Teichm\"{u}ller character $\omega$
\begin{equation}\label{frobeniusteichmuller}
\widetilde F(t):=\operatorname{lim}_{m\to\infty} F_m(t)=\omega(1+t)-1.    
\end{equation}
On the other hand, we can evaluate the asymptotic discrepancy between $F_m$ and the multiplication by $p^m$, which can be viewed as a limit of Verschiebung operators since $VF=FV$ is the multiplication by $p$. If $u$ is divisible by $p$,  the limit of the quotient of $F_m(u)$
by $p^m$ is exactly the logarithmic function  $\operatorname{log}(1+u)$ according to Corollary~\ref{cor:loglim},
\begin{equation}\label{eq:loglink}
\operatorname{log}(1+u)=\operatorname{lim}_{m\to\infty} \frac{1}{p^m}\left((1+u)^{p^m} -1\right).     
\end{equation} 
Note here we only use the expression of $F_m$ rather than its mapping.

At this point, we understand what the canonical Frobenius lifting really represents in $p$-adic Teichm\"{u}ller theory. In fact, it is a lattice consisting of two directions: on the one hand, we can decompose any element $a\in W(\mathbb F_p)\cong\ZZ_p$ into Witt vectors  $a=\sum_{i\geq0}p^i\omega(a_i)$ with $a_i\in\mathbb F_p$,  
the Frobenius liftings $F_m$ generate  horizontal arrows of the lattice  $\mathcal{CFL}$ (canonical Frobenius lifting),
\begin{equation}\label{canonicalliftinghoriz0}
\cdots\stackrel{F_{i_{m-2}}}{\longrightarrow}\mathcal{CFL}_{m-1,n}\stackrel{F_{i_{m-1}}}{\longrightarrow}\mathcal{CFL}_{m,n}\stackrel{F_{i_{m}}}{\longrightarrow}\mathcal{CFL}_{m+1,n}\stackrel{F_{i_{m+1}}}{\longrightarrow}\cdots;    
\end{equation}
in addition, the limits $\widetilde F$ of Frobenius lifting $F_m$ also generate  horizontal arrows of the lattice  $\mathcal{CFL}$,
\begin{equation}\label{canonicalliftinghoriz}
\cdots\stackrel{\widetilde F}{\longrightarrow}\mathcal{CFL}_{m-1,n}\stackrel{\widetilde F}{\longrightarrow}\mathcal{CFL}_{m,n}\stackrel{\widetilde F}{\longrightarrow}\mathcal{CFL}_{m+1,n}\stackrel{\widetilde F}{\longrightarrow}\cdots;     
\end{equation}
on the other hand, we have $a=a_f+a_p$ with $a_f\in \mathbb F_p$ and $a_p\equiv 0 \; \operatorname{mod} p$, then $\operatorname{log}(1+a_p)$ given by equation~\ref{eq:loglink}
generates vertical arrows of the lattice  $\mathcal{CFL}$, 
\begin{equation}\label{canonicalliftingvert}
\cdots\stackrel{\operatorname{log}}{\longrightarrow}\mathcal{CFL}_{m,n-1}\stackrel{\operatorname{log}}{\longrightarrow}\mathcal{CFL}_{m,n}\stackrel{\operatorname{log}}{\longrightarrow}\mathcal{CFL}_{m,n+1}\stackrel{\operatorname{log}}{\longrightarrow}\cdots.     
\end{equation}
There are no extra efforts needed to generalize all above discussions to a finite perfect field $k$ of cardinality number $p^f$. 
The one more thing required to do is to replace the number $p$ by $p^f$ and use Proposition~\ref{numberfieldomega}.

%\subsection{$p$-curvature}

\section{Inter-universe Teichm\"{u}ller theory}\label{sec:interuniverseteich}
In light of the  close relationship between $p$-adic Teichm\"{u}ller theory and $p$-adic diffeomorphism group, it is quite natural to put the Inter-universe Teichm\"{u}ller theory (IUT) in a framework with the interaction of $p$-adic diffeomorphism groups. According to Mochizuki~\cite{mochizuki12}, IUT is an arithmetic Teichm\"{u}ller theory for number fields equipped with an elliptic curve. To start with, we view such an object as an arithmetic surface with group structure, or more generally a group scheme, and present its elementary properties following the same lines as in~\cite[Chapter IV]{silverman1994}.

\subsection{Arithmetic surface and group scheme}
An arithmetic surface $C$ is a one-parameter family of one-dimensional varieties, so it has the name surface for being a two-dimensional space.
\begin{defi}
Let $R$ be a Dedekind domain with fraction field $K$. An arithmetic surface over $R$ is an $R$-scheme $\mathcal C$ whose generic fibre  is a non-singular connected projective curve $C/K$ and whose special fibres are bunches of curves over appropriate residue fields $k$.   
\end{defi}
In particular, $R$ is often defined as a discrete valuation ring to study the reduced cases through localization of a Dedekind domain. In addition, it is of particular interest for arithmetic surfaces which are regular, or proper, or smooth over $R$. 

If the generic fibre $C/K$ is endowed with a group structure, then it is a one-dimensional connected group variety, which can be classified into three cases: either an addition group $\mathbb G_a$, a multiplicative group $\mathbb G_m$ or an elliptic curve; see~\cite[Theorem 1.6]{silverman1994}. An arithmetic surface $C$ is called a group scheme over $R$ if all of its fibres form an algebraic family of groups.
\begin{expl}
Assume that  $K$ is a number field, then its ring of integers $R=\mathcal{O}_K$ is a Dedekind domain. Let 
$E/K$ be any elliptic curve, and let 
$\mathfrak p$ be any prime (maximal) ideal of $R$. A minimal Weierstrass equation for $E$ is given by
\begin{equation}
y^2+a_1xy+a_3y=x^3+a_2x^2+a_4x+a_6,    
\end{equation}
where $a_i\in R$. So we can use this equation to define an $R$-scheme $\mathcal E\subset \mathbb P^2_R$. The special fibre $\mathcal E_{\mathfrak p}$ is an elliptic curve over the residue field $R/\mathfrak p$. We can further consider the localization of $R$ at  $\mathfrak p$. Then the localization  
$\mathcal E\times_R  R_{(\mathfrak p)}$ is a group scheme over
a discrete valuation ring $R_{(\mathfrak p)}$, and the corresponding (possibly infinitely many) special fibres are reduced to the unique one
at $\mathfrak p$, which is elliptic curve $E$ over $R/\mathfrak p\cong R_{(\mathfrak p)}/\mathfrak p$. On the other hand, we denote by $R_{\mathfrak p}$ the completion of $R$ via the $\mathfrak p$-adic valuation $v_{\mathfrak p}$. Similar to localization, the completion $\mathcal E\times_R  R_{\mathfrak p}$ is a group scheme over $R_{\mathfrak p}$ with special fibre defined as elliptic curve $E$ over $ R_{\mathfrak p}/\mathfrak p\cong R/\mathfrak p$, and generic fibre defined as elliptic curve $E$ over the $\mathfrak p$-adic field $ K_{\mathfrak p}$.
\end{expl}
According to~\cite[Proposition 1.5]{silverman1994}, the connected component of a group variety $G$ which contains the identity element is a normal subgroup of $G$ of finite index. So we can focus on
the identity component $G^0$ without much loss of generality.

\begin{rem}
The notion of a group scheme can be seen as an algebraic counterpart of a principal bundle: the fibres of both sides have a group structure; and a scheme is an algebraic analogue of a manifold. Intuitively, a scheme $X$ is regular if every point
of $X$ has a tangent space of the correct dimenison; a morphism of schemes $X\to S$ is proper if all of its fibres are complete and separable, which are analogues of compact and Hausdorff; a morphism $X\to S$ is called smooth if all of its fibres are non-singular.
\end{rem}
This similarity between a group scheme and a principal bundle reminds us of a corresponding exact sequence of automorphisms in a group scheme, as has been shown in formula~\ref{padicexactseq}
for a $p$-adic principal bundle.

\subsection{Automorphisms of a group scheme}
Let  $K$ be a number field, and let $\mathcal E$ be a one-dimensional group scheme over $R$, i.e., the ring of integers of $K$. Then the fibre $\mathcal E_{\mathfrak a}$ of $\mathcal E$ is either an addition group $\mathbb G_a$, a multiplicative group $\mathbb G_m$ or an elliptic curve $E$ as mentioned above. A na{i}ve guess gives us an non-exact sequence
\begin{equation}\label{schemeexactseq1}
\operatorname{Aut}(\mathcal E_{\mathfrak a})\hookrightarrow \operatorname{Aut}(\mathcal E)\twoheadrightarrow \operatorname{Hom}(R),    
\end{equation}
where $\operatorname{Hom}(R)$ may refer to some sort of homeomorphism group of $R$ (or $\operatorname{Spec}(R)$). Note that this cannot be an exact sequence in the classical sense, since different points on $\operatorname{Spec}(R)$ correspond to totally non-isomorphic fibres. Moreover, it is obviously a problem that no analytic techniques can be applied to the right side of the exact sequence. Hence it may be more convenient to
consider the completion of group scheme $\widetilde{\mathcal E}:=\mathcal E\times_R  R_{\mathfrak p}$, and then the non-exact sequence becomes
\begin{equation}\label{schemeexactseq2}
\operatorname{Aut}(\widetilde{\mathcal E_{\mathfrak a}})\hookrightarrow \operatorname{Aut}(\widetilde{\mathcal E})\twoheadrightarrow \operatorname{Diff}^{\operatorname{an}}(R_{\mathfrak p}),    
\end{equation}
where $\widetilde{\mathcal E_{\mathfrak a}}$ is a generic fibre (resp. special fibre) at $0$ (resp. $\mathfrak p$). Also note that for the trivial case $K=\mathbb Q$, $R_\mathfrak p$ is just the unit $p$-adic ball $B_{\epsilon}$, so the sequence~\ref{schemeexactseq2} is a generalization of $p$-adic exact sequence~\ref{padicexactseq}.

\begin{rem}
It is important to see that there exist two sorts of topology on a group scheme over a complete field. One is the coarse Zariski topology with respect to the morphisms between $R$-schemes and their corresponding fibre bundle structure; the other is the refined $p$-adic topology with respect to the local analytic properties. This is similar to the situation one will face in infinite dimensional Riemannian geometry: the Banach manifold structure (e.g., Sobolev metrics $H^k$ or $W^{k,r}$ for $k\geq1$) induces strong topology in local neighborhood, while it is the weak Riemannian metric (e.g., $L^r$-metric) on the whole infinite dimensional manifold that are of interest in hydrodynamics; see e.g.~\cite{arnold1966geometrie,Arnold98}.  
\end{rem}

First it is not hard to see that the automorphism groups are of the following forms if the fibre is additive group $\mathbb G_a$ or multiplicative group $\mathbb G_m$. 
\begin{expl}
Let $\mathcal E$ be a group scheme with fibre equal to $\mathbb G_a$.
The automorphism groups of generic fibre and special fibre are given by 
\begin{equation}
\operatorname{Aut}(\widetilde{\mathcal E_{\mathfrak (0)}})\cong \mathbb G_m, \; \; 
\operatorname{Aut}(\widetilde{\mathcal E_{\mathfrak p}})\cong \operatorname{GL}_f(\mathbb F_p).
\end{equation}
In a similar manner, if $\mathcal E$ is a group scheme with fibre equal to $\mathbb G_m$, automorphism groups of generic fibre and special fibre are given by 
\begin{equation}
\operatorname{Aut}(\widetilde{\mathcal E_{\mathfrak (0)}})\cong \ZZ\times R^\times, \; \; 
\operatorname{Aut}(\widetilde{\mathcal E_{\mathfrak p}})\cong \mathbb F_{p^f}^\times,
\end{equation}
where $p^f$ is the cardinality of the residue field $R/\mathfrak p$, and "$\times$" defines the group action on $\mathbb F_{p^f}$ as field multiplication. The case of affine group  $\operatorname{A}(\mathbb Q_p)$ discussed in Proposition~\ref{affinegroupiso} is quite similar to these two cases, the difference comes along with the appearance of translation.  
\end{expl}
Our main goal is to understand the structure of automorphism group $\operatorname{Aut}(\widetilde{\mathcal E_{\mathfrak a}})$ for $\mathcal E$ being a group scheme with fibre being an elliptic curve $E$. But the automorphism group as a projective curve is too large for it contains every point of the elliptic curve; 
and the automorphism group  fixing the identity is too small as a finite group of order dividing $24$, see e.g.~\cite[III~10.1]{silverman1986}. So it is not convenient to deal with these two groups. One way to overcome this problem is to introduce the $l$-torsion of $E[l]\cong\ZZ/l\times\ZZ/l$ for some positive integer $l$ prime to $p$.
We define $\widetilde{\mathcal E}[l]$  as the $l$-torsion of $\widetilde{\mathcal E}$, which gives $E[l]$ in each fibre. Therefore, we have the following sequence of $R_{\mathfrak p}$-schemes
\begin{equation}\label{schemeexactseq3}
\widetilde{\mathcal E_{\mathfrak a}}[l^n]\hookrightarrow \widetilde{\mathcal E}[l^n]\twoheadrightarrow \operatorname{Spec}(R_{\mathfrak p}),    
\end{equation}
and the corresponding fibration of automorphism groups
\begin{equation}\label{schemeexactseq3*}
\operatorname{Aut}(\widetilde{\mathcal E_{\mathfrak a}}[l^n])\hookrightarrow \operatorname{Aut}(\widetilde{\mathcal E_{\mathfrak a}}[l^n])\twoheadrightarrow \operatorname{Diff}^{\operatorname{an}}(R_{\mathfrak p}).    
\end{equation}
With the inverse limit being taken with respect to the natural map 
\begin{align*}
E[l^{n+1}] \stackrel{[l]}{\longrightarrow}E[l^{n}],    
\end{align*}
 we obtain a fibration of 
 $R_{\mathfrak p}$-schemes
 with Tate modules as its fibre
\begin{equation}\label{schemeexactseq4}
T_l(\widetilde{\mathcal E_{\mathfrak a}})\hookrightarrow T_l(\widetilde{\mathcal E})\twoheadrightarrow \operatorname{Spec}(R_{\mathfrak p}),   
\end{equation}
where $T_l(\widetilde{\mathcal E_{\mathfrak a}})$ is
the $l$-adic Tate module of $E$ at the point $\mathfrak a$. Finally, the desired fibration of $\operatorname{Aut}(T_l(\widetilde{\mathcal E}))$ is given by 
\begin{equation}\label{schemeexactseq5}
\operatorname{Aut}(T_l(\widetilde{\mathcal E_{\mathfrak a}}))\hookrightarrow \operatorname{Aut}(T_l(\widetilde{\mathcal E}))\twoheadrightarrow \operatorname{Diff}^{\operatorname{an}}(R_{\mathfrak p}).   
\end{equation}
Note that here we regard the automorphism of a group scheme as a diffeomorphism (in $\mathfrak p$-adic topology) that induces an automorphism  
at each of its fibre, and we only consider the automorphisms covering the connected component of the identity. This coincides with the definition of an automorphism for a principal bundle in Abbati et. al.~\cite{abbati1989}, which we discussed in the previous section.

On the other hand, it is well-known that the absolute Galois group $G_{\overline L/L}$ acts on $E[l]$ for elliptic curve $E/L$. Let $l\neq p$ be any prime number, we have a representation 
\begin{equation}
 G_{\overline L/L}\longrightarrow\operatorname{Aut}(E[l])\cong \operatorname{GL}_2(\ZZ/l),   
\end{equation}
which further induces the $l$-adic representation
\begin{equation}
 \rho_l:G_{\overline L/L}\longrightarrow\operatorname{Aut}(T_l(E)) \cong\operatorname{GL}_2(\ZZ_l).  
\end{equation}
If $L=K_{\mathfrak p}$, this gives the $l$-adic representation of the generic fibre of $\widetilde{\mathcal E}$. 
In addition, we obtain the $l$-adic representation of the special fibre 
\begin{equation}
 \overline\rho_l:G_{\overline k/k}\longrightarrow\operatorname{Aut}(T_l(E)) \cong\operatorname{GL}_2(\ZZ_l),  
\end{equation}
which factors through an 
inertia group $I_v$
using the following quotient relation with respect to the residue field $k$ of $L$
\begin{equation}
G_{\overline L/L}/ I_v\cong G_{\overline k/k}. 
\end{equation}

\subsection{The log-structure and log-volume}\label{subsec:logstructure}
One can easily see that the appearance of 
perturbations of the identity $\operatorname{Diff}^{\operatorname{an}}(R_{\mathfrak p})$ in sequence~\ref{schemeexactseq2} and~\ref{schemeexactseq5}
yields a log-structure on $\operatorname{Aut}(\widetilde{\mathcal E})$ and $\operatorname{Aut}(T_l(\widetilde{\mathcal E}))$, which can be derived from a Frobenius lifting argument on diffeomorphism group or directly from a Frobenius approximation for $p$-adic field in equation~\ref{cor:loglim}.
To better understand this log-structure, we first recall some facts and calculations from elementary field theory.

Let $K$ be a number field with the ring of integers $R:=\mathcal O_{K}$. For any prime ideal $\mathfrak p$ of $R$, the $\mathfrak p$-adic field $K_{\mathfrak p}$ is totally ramified over $\mathbb Q_p$ with ramification index $e$. Then there exists an isomorphism of $\ZZ_p$-algebras $\ZZ_p[x]/(g(x))\to R_{\mathfrak p}$, where $g(x)\in \ZZ_p[x]$ is a monic polynomial such that $g(x)\equiv x^e (\operatorname{mod}p)$ and the monomial $x$ is mapped to some uniformizer $\pi$ of $R_{\mathfrak p}$, which has order $\frac{1}{e}$ in terms of the $p$-adic valuation. This algebraic construction can be seen as a consequence of a $p$-adic diffeomorphism.
\begin{prop}\label{prop: rootdiff}
%$\operatorname{log}\epsilon\leq -\frac{1}{e-1}\operatorname{log}p$  
Let the ramification number $e>1$ and radius $\epsilon<1$,  and let $\varphi=\operatorname{id}+g$. Then $\varphi$ is a $\mathfrak p$-adic diffeomorphism in $B_\epsilon\subset R_{\mathfrak p}$, which takes its fixed points exactly at $\pi$ and its conjugates, i.e., the roots of $g$ in $R_{\mathfrak p}$.
\end{prop}
\begin{proof}
By Theorem~\ref{thmdifflie}, it suffices to show that $g$ is a Lipschitz function in $B_\epsilon$ with  Lipschitz constant less than one. Indeed, for any $x,y\in B_\epsilon$, we have 
\begin{align*}
|g(x)-g(y)|=|x^e-y^e|\leq |x-y||x^{e-1}+x^{e-2}y+\cdots+y^{e-1}|\leq \epsilon^{e-1}|x-y|,     
\end{align*}
which concludes the proof since $\epsilon^{e-1}<1$.
\end{proof}

From  elementary ramification theory, we know that the different  $\mathcal D_{K/\mathbb Q_p}$ is the principal ideal $(D_xg(\pi))$. Under the same condition as in Proposition~\ref{prop: rootdiff}, we obtain a charaterization of the different by $p$-adic diffeomorphism group and its log-structure.
\begin{cor}\label{different}
Recall that the  $p$-adic logarithm on $\operatorname{Diff}^{\operatorname{an}}(B_\epsilon)$ is given by $\Phi_\infty:\varphi\to\operatorname{log}(D_x\varphi)$. In this case, the order of $\Phi_\infty(\varphi)$ at the uniformizer $\pi$ coincides with the order of the different $\mathcal D_{K/\mathbb Q_p}$.  
\end{cor}
 \begin{proof}
This is a simple application of Theorem~\ref{thm:isometry:infty}.     
 \end{proof}
Note that all the above discussions can be extended to an arbitrary finite field extension of number fields $K'/K$. More precisely, if $\mathfrak p'$ is a prime ideal of $R':=\mathcal O_{K'}$ lying over a prime ideal $\mathfrak p$ of $R:=\mathcal O_{K}$, then there exists an isomorphism of $R_{\mathfrak p}$-algebras $R_{\mathfrak p}[x]/(g(x))\to R'_{\mathfrak p'}$, and we can extend the preceding proposition and corollary to a general field extension as far as the case of $K/\mathbb Q_p$ has shown us. In IUT~\cite[IV Proposition~1.3]{mochizuki12}, this method of reducing to the case of a totally ramified extension was used to estimate the order of differents.

\begin{rem}
It is natural to question whether the above diff-group theoretic construction is compatible with the algebraic-number theoretic construction, unfortunately the answer is negative. Let $K'_{\mathfrak p'}/K_\mathfrak p/\mathbb Q_p$ be the field extensions with the descending primes $\mathfrak p'\supset\mathfrak p\supset p$, and let $\varphi$ and $\phi$ denote the corresponding $p$-adic diffeomorphisms.
On the one hand, we have the additive relation
\begin{equation}
\Phi_\infty(\varphi\circ\phi)=\Phi_\infty(\varphi)\circ\phi+\Phi_\infty(\phi),   
\end{equation}
while on the other hand, 
we have the multiplicative relation of differents
\begin{equation}
\mathcal D_{K'_{\mathfrak p'}/\mathbb Q_p}=\mathcal D_{K'_{\mathfrak p'}/K_\mathfrak p}
(\mathcal D_{K_{\mathfrak p}/\mathbb Q_p} R'_{\mathfrak p'}).   
\end{equation}
\end{rem}

The $p$-adic logarithmic transformation $\Phi_\infty$ maps the group $\operatorname{Diff}^{\operatorname{an}}(R_{\mathfrak p})$ to a log-shell in the sense of IUT. In order to evaluate this log-shell,  we follow the same steps in Mochizuki's absolute anabelian geometry~\cite[Proposition~5.7, 5.8]{mochizuki15}, and define the log-volume. 

\begin{defi}
Let $K$ be a number field with absolute degree $n:=[K_\mathfrak p:\mathbb Q_p]$. The log-volume  $\mu_{K}^{\operatorname{log}}(\cdot)$ 
on compact open subsets of ring of integers $R_\mathfrak p$ is defined as the logarithm of the normalized Haar measure with respect to $\mathfrak p$-adic valuation $v_{\mathfrak p}$. In addition, it can be normalized to the normalized log-volume  $\mu^{\operatorname{log}}(\cdot)$ so that $\mu^{\operatorname{log}}(R_{\mathfrak p})=0$ and $\mu^{\operatorname{log}}(p R_{\mathfrak p})=-\operatorname{log}p$. Here the symbol "log" denotes the usual logarithm valued in $\RR$ (not $p$-adic logarithm).
\end{defi}
 
Therefore, the log-volumes are of the form 
\begin{equation}
\mu_K^{\operatorname{log}}(\cdot)=\operatorname{log(\mu_K(\cdot))}= n\mu^{\operatorname{log}}(\cdot)=n\operatorname{log(\mu(\cdot))},   
\end{equation}
since we have 
$\mu_K^{\operatorname{log}}(p R_{\mathfrak p})= -\operatorname{log}([R_{\mathfrak p}:p R_{\mathfrak p}])= -n\operatorname{log}p$.

\begin{expl}\label{logvolume}
We compute several examples of log-volumes from~\cite{mochizuki12,mochizuki15}.
\begin{itemize}
\item[1.] Suppose that for any prime ideal $\mathfrak p$ of $R$, the $\mathfrak p$-adic field $K_{\mathfrak p}$ is ramified over $\mathbb Q_p$ with ramification index $e$ and the cardinality
$p^f$ of the residue field of $k$. Then the log-volume of $\mathfrak p R_{\mathfrak p}$ and the multiplicative group can be computed as follows
\begin{align*}
\mu_K^{\operatorname{log}}(1+\mathfrak p R_{\mathfrak p})=\mu_K^{\operatorname{log}}(\mathfrak p R_{\mathfrak p})=-\operatorname{log}([\mathcal O_k:\mathfrak p R_{\mathfrak p}])=-f\operatorname{log}p.
\end{align*}
Moreover, the quotient relation $R^\times_{\mathfrak p}/(1+\mathfrak p R_{\mathfrak p})\cong k^\times$ implies that   
\begin{align*}
\mu_K^{\operatorname{log}}(R^\times_{\mathfrak p})=\mu_K^{\operatorname{log}}(1+\mathfrak p R_{\mathfrak p})+\operatorname{log}(|k^\times|)=\operatorname{log}(1-p^{-f}).
\end{align*}

\item[2.]
From  standard results of ramification theory, we have the equality $n=ef$, and that the torsion subgroup of multiplicative group $R^\times_{\mathfrak p}$ is given by 
\begin{align*}
R^{\mu}_{\mathfrak p}=k^{\times}\times\mu_{p^m},    
\end{align*}
where $k^{\times}$ (resp. $\mu_{p^m}$) is a cyclic group of order $p^{f}-1$ (resp. $p^m$). Consequently, the log-volume of quotient group $R^{\times\mu}_{\mathfrak p}:=R^\times_{\mathfrak p}/R^{\mu}_{\mathfrak p}$ is given by
\begin{align*}
\mu_K^{\operatorname{log}}(R^{\times\mu}_{\mathfrak p})=\mu_K^{\operatorname{log}}(R^{\times}_{\mathfrak p})-\operatorname{log}(|R^{\mu}_{\mathfrak p}|)=-(m+f)\operatorname{log}p.    
\end{align*}
\item[3.] Either in IUT~\cite{mochizuki12} or in absolute anabelian geometry~\cite{mochizuki15}, the torsion subgroup was not taken into consideration for the volume of log-shell, the only part determining the log-volume is the quotient group $R^{\times\mu}_{\mathfrak p}$. Consequently in Mochizuki's sense, the log-volume of log-shell is simply given by 
\begin{align*}
\mu^{\operatorname{log}}(\operatorname{log}_p(R^{\times}_{\mathfrak p}))=\mu^{\operatorname{log}}(R^{\times\mu}_{\mathfrak p})=-(\frac {m}{ef}+\frac{1}{e})\operatorname{log}p,    
\end{align*}
where $\operatorname{log}_p$ denotes $p$-adic logarithm in contrast with the usual logarithm. 
\end{itemize}    
\end{expl}
In fact, we can generalize the $p$-adic logarithmic transformation onto log-shell $\operatorname{log}_p(R_\mathfrak p^\times)$ by using an alternative way of defining logarithm 
that was introduced in Corollary~\ref{cor:loglim}. Note that it is applicable since this corollary only requires  $x=1+z$ with $|z|_p<1$. In addition, Proposition~\ref{teichmullerchar} can be extended to  an  arbitrary $K_\mathfrak p$ for any number field $K$.
\begin{prop}\label{numberfieldomega}
Let  $x\in R_\mathfrak p$ be an element in the residue class $\tilde x\equiv\tilde a\ (\operatorname{mod} \mathfrak p)$ for $a\in\{1,2,\cdots,p^f-1\}$.  Then $x^{p^{fm}}$ converges to  $\omega(\tilde a)$ as $m\to\infty$. 
\end{prop}
\begin{proof}
This is a direct consequence of the generalized Euler's theorem, i.e., $x^{N(\mathfrak p)^{m}}\equiv x^{N(\mathfrak p)^{m-1}}(\operatorname{mod} \mathfrak p)$. The remaining part of proof is similar to that of Proposition~\ref{teichmullerchar}.    
\end{proof}
Then the generalized logarithmic function in $D_a:=\{a+z: |z|_p<1\}$
is given by 
\begin{equation}
\widetilde{\operatorname{log}}(x):= \operatorname{lim}_{m\to\infty}\frac{1}{p^{fm}}\left(x^{p^{fm}} -a^{p^{fm}}\right)=\omega(\tilde a)\operatorname{log}(x/a),    
\end{equation}  
where $\omega$ is the Teichm\"{u}ller character with respect to  $R_\mathfrak p$.
Following the same steps, we have a family of generalized logarithmic functions as shown in~\ref{genlogqpcharacter}\begin{equation}
\widetilde{\operatorname{log}}_l(x):=s\operatorname{log}(p)+\omega^l(\tilde a)\operatorname{log}(g),   
\end{equation}
where  $x\in  R_\mathfrak p^\times$ is expressed uniquely as $x=p^sag$  for $s\in\ZZ$ and $g\in 1+p\ZZ_p$,  $a,l\in\{1,2,\cdots,p^f-1\}$, and also note that  $s\operatorname{log}(p)$ is a term of formal nature.

\subsection{Moser's trick and Galois cohomology}
In differential geometry, Moser's trick~\cite{moser1965} is a method to link two differential forms $\alpha$ and $\beta$ on a compact smooth manifold by a diffeomorphism $\varphi\in\operatorname{Diff}(M)$ such that $\varphi^*\alpha=\beta$.
By Moser's trick, the action of $\operatorname{Diff}(M)$ on the space of probability densities $\operatorname{Prob}(M)$ is transitive, and the kernel is exactly the infinite dimensional group $\operatorname{Diff}_\mu(M)$ of volume-preserving diffeomorphisms (volumorphisms) of $M$, thus we have
\begin{equation}\label{probequivdiff}
\operatorname{Prob}(M)\cong\operatorname{Diff}(M)/\operatorname{Diff}_\mu(M),
\end{equation}
where $\operatorname{Diff}_\mu(M)$ acts by left cosets. The quotient relation~\ref{probequivdiff} bridges two totally different fields: the space of probability densities $\operatorname{Prob}(M)$ that consists of positive volume forms of total volume one that reside in the top dimensional De Rham cohomology $H_{\operatorname{dR}}^n(M)$, and on the other side a quotient of two  Fr\'echet Lie groups.

This enlightens us of a strict relationship between cohomology groups and infinite dimensional diffeomorphism groups. Indeed, this is the case for certain $p$-adic group cohomology and  $p$-adic
diffeomorphism groups. First recall the notion of a group cohomology.
\begin{defi}
Let $G$ be a group and let $N$ be a $G$-module. The group of $1$-cochains $C^1(G,N)$ is denoted by the set of all mappings $\xi: G\to N$. The group of $1$-cocycles $Z^1(G,N)$ is denoted by the set of all cochains $\xi$ such that 
\begin{equation*}
\xi_{\sigma\tau}=\xi_{\sigma}^\tau+\xi_{\tau} \; \; \text{for all} 
\;\sigma,\tau\in G.   
\end{equation*}
The group of $1$-coboundaries $B^1(G,N)$ is denoted by the set of all cochains $\xi$ with an element $n\in N$ such that 
\begin{equation*}
\xi_{\sigma}=n^\sigma-n \; \; \text{for all} 
\;\sigma\in G.   
\end{equation*}
The first cohomology group is then the quotient group 
\begin{equation*}
H^1(G,N):=Z^1(G,N)/B^1(G,N).    
\end{equation*}
\end{defi}
Now we consider a $p$-adic completion $K$ of an arbitrary number field and a $p$-adic ball $B_\epsilon$ in this field. Then the space of locally analytic functions $C^{\operatorname{an}}(B_\epsilon)$ is a  $\operatorname{Diff}^{\operatorname{an}}(B_\epsilon)$-module, since the group action can operate from the right side of an element in $C^{\operatorname{an}}(B_\epsilon)$.  It is obvious that the $0$-th cohomology group of $C^{\operatorname{an}}(B_\epsilon)$ is the set of all $\operatorname{Diff}^{\operatorname{an}}(B_\epsilon)$-invariant functions in $C^{\operatorname{an}}(B_\epsilon)$, and this set is non-zero since it includes all the constants.
\begin{prop}
The mappings $\xi:\operatorname{Diff}^{\operatorname{an}}(B_\epsilon)\to C^{\operatorname{an}}(B_\epsilon)$  via: $\varphi\to\varphi-\operatorname{id}$ define a subgroup of $1$-cocycles. Moreover, if $\varphi_i=\operatorname{id}+f_i$ are diffeomorphisms in $\operatorname{Diff}^{\operatorname{an}}(B_\epsilon)$, then the corresponding cocycle $\xi_{\widetilde\varphi}$ of their composition $\widetilde\varphi:=\prod\varphi_i$ is the sum of $f_i$ in the first cohomology group $H^1(\operatorname{Diff}^{\operatorname{an}}(B_\epsilon),C^{\operatorname{an}}(B_\epsilon))$.
\end{prop}
\begin{proof}
Let $\varphi=\operatorname{id}+f$ and $\phi=\operatorname{id}+g$ denote two $p$-adic diffeomorphisms in the  $p$-adic ball $B_\epsilon$. The claim on $\xi$ is immediate from
the calculation
\begin{equation}
\xi_{\varphi\circ\phi}=   \varphi\circ\phi-\operatorname{id}=f\circ\phi+g=\xi_{\varphi}^\phi+\xi_{\phi}. 
\end{equation}
Therefore, within the cohomology group, we have 
$\xi_{\varphi\circ\phi}\equiv f+g$ and the second claim follows by induction.
\end{proof}
This proposition reveals a logarithmic structure in the 
cohomology group, in other words, the cocycle $\xi$ maps a multiplication of the diffeomorphism group into 
an addition. Furthermore, this proposition has a finite dimensional version that is closely associated to the Galois cohomology. By proposition~\ref{prop: rootdiff}, if the radius $\epsilon<1$, any monic minimal polynomial $g$ corresponds to a diffeomorphism $\varphi=\operatorname{id}+g$ in $B_\epsilon$, which has fixed points exactly at the roots of $g$. But we know that the Galois group $G_{L/K}$ is 
an automorphism group that fixes all elements of $K$ for the extension $L=K[x]/(g(x))$.
Henceforth, we have a larger automorphism group $\mathcal A_{L/K}$ by an enhancement of Galois group $G_{L/K}$.
\begin{equation}
 1\to G_{L/K}\to \mathcal A_{L/K}\to \langle\varphi\rangle\to1,   
\end{equation}
where $\langle\varphi\rangle$
is a subgroup of $\operatorname{Diff}^{\operatorname{an}}(B_\epsilon)$ generated by $\varphi$. Denote by $\mathcal A$ the group generated by all such 
groups $\mathcal A_{L/K}$
for finite field extensions $L/K$. Then we
obtain the following sequence of mappings, 
\begin{equation}\label{galoiscohom}
 G_{\overline K/K}\stackrel{i}\hookrightarrow \mathcal A\stackrel{\pi}\rightarrow \operatorname{Diff}^{\operatorname{an}}(B_\epsilon),
\end{equation}
where $i$ is the injective from the absolute Galois group $G_{\overline K/K}$ and 
$\pi$ is the projection to a subset of $\operatorname{Diff}^{\operatorname{an}}(B_\epsilon)$ containing all monic irreducible polynomials. The sequence~\ref{galoiscohom} implies that the first cohomology group (from  $\mathcal A$ to $C^{\operatorname{an}}(B_\epsilon)$)  can be considered as the combination of Galois cohomology $H^1(G_{\overline K/K}, C^{\operatorname{an}}(B_\epsilon))$  and group cohomology $H^1(\operatorname{Diff}^{\operatorname{an}}(B_\epsilon),C^{\operatorname{an}}(B_\epsilon))$.

\subsection{Tate curves and $p$-adic theta functions}
In the previous subsection, we apply some of the techniques developed in Section~\ref{sec:padicdiffgroup} to diffeomorphism group $\operatorname{Diff}^{\operatorname{an}}(R_\mathfrak p)$ over an arbitrary number field, which corresponds to the right hand side of the sequence~\ref{schemeexactseq5}. Now we turn to the left hand side, and consider the torsion points $E[l]$ of a given elliptic curve $E$, which will give the \'etale picture in IUT. First we recall the notion of a Tate curve and Tate's $p$-adic uniformization theorem. 

Let $E$ be an elliptic curve over $p$-adic fields $K/\mathbb Q_p$. Tate showed that for every $q\in K^\times$ with $|q|<1$, there is an elliptic curve $E_q/K$ that is $p$-adic analytically isomorphic to $K^\times/q^{\ZZ}$. More explicitly, there exists for every $q,u$ in $K^\times$ with $|q|<1$, an injective homomorphism of groups with image $E_q\subset \mathbb P_2(K)$
\begin{equation}
\phi_q: \begin{cases}
K^\times/q^\ZZ  	&\to  \mathbb P_2(K)\\ 
u 						&\mapsto 
(1:X(u,q):Y(u,q)),
\end{cases}
\end{equation}
where $X$ and $Y$ are coordinates for Weierstrass equation; see e.g.,~\cite{brown2003, silverman1994}. One way to prove such uniformization is to use the theory of $p$-adic theta functions, as is shown by Brown's nice exposition~\cite{brown2003}. Since $p$-adic theta functions play a central role in IUT, we carefully introduce this notion following the exposition. 
\begin{defi}
A function $f:K^\times\to K$ is said to be holomorphic if it is defined by a series
$f(u)=\sum_{n=-\infty}^{\infty}a_n u^n$, with coefficients $a_i\in K$ such that $f$ is convergent for any $u\in K^\times$.
Let $H_K$ denote the field of holomorphic functions in $K$.  Let $M_K$ denote the field of meromorphic functions in $K$, i.e., the quotient field of $H_K$. Let $M_{K,q}$ denote the subfield of $M_K$ consisting of all meromorphic functions which are $q$-periodic for a fixed $q\in K^\times$ with $|q|<1$.
\end{defi}
It is not hard to see that the intersection $H_K\cap M_{K,q}$ is exactly the set of  constants. This is coincident with the complex case for which every holomorphic elliptic function is constant. 
\begin{defi}
Let $c\in K^\times, r\in \ZZ$. A theta function of type $cu^r$ is a holomorphic function $f$  satisfying 
\begin{equation}
f(u)=cu^rf(qu)    
\end{equation}
for all $u\in K^\times$. We denote by $H_{K,q}(cu^r)$ the space of all such theta functions of type $cu^r$.
The fundamental theta function is the theta function of type $-u$ defined by 
\begin{equation}
\theta(u)=(1-u)\prod_{n=1}^{\infty}\left(1-q^nu\right)\left(1-q^nu^{-1}\right),   
\end{equation}
with zeros exactly on $q^\ZZ$. Let $c\in K^\times$,  $\theta_c$ is defined by $\theta_c(u):=\theta(u/c)$. Then $\theta_c$ is a theta function of type $-c^{-1}z$ with zeros exactly on $cq^\ZZ$.
\end{defi}
\begin{rem}
The theta function frequently used in IUT is of the form 
\begin{equation}
\tilde\theta(\tilde u):=q^{-\frac{1}{8}}\sum_{n=-\infty}^{\infty} (-1)^n  q^{\frac{1}{2}(n+\frac{1}{2})^2}\tilde u^{2n+1},  
\end{equation}
where $\tilde u:=\sqrt{u}$ is a square root element. It is "$q$-periodic" in the sense that 
\begin{equation}\label{qperiodic}
\tilde\theta(\tilde u)=(-1)^a q^{\frac{a^2}{2}} \tilde u^{2a} \tilde\theta(q^{\frac{a}{2}}\tilde u),  
\end{equation}
for $a\in\ZZ$. Then
$\tilde\theta$ is associated with the fundamental theta function $\theta$ by the following equality
\begin{equation}\label{thetarelation}
\tilde u \tilde\theta(\tilde u)=-\eta(q)\theta(u),    
\end{equation}
where $\eta(q):=\prod_{n=1}^{\infty}(1-q^n)$ is the $p$-adic eta function. This is why in IUT the $2l$-th torsion point of the theta function $\tilde\theta$ can be seen as the $l$-th torsion point of $\theta$.
\end{rem} 

There is a $p$-adic version of the Abel-Jacobi theorem, which is useful in many aspects.
\begin{thm}[Abel-Jacobi]\label{abeljacobithm}
For any two sets $\{a_1,\cdots,a_k\}$, $\{b_1,\cdots, b_k\}$ of $k$ integers in $K^\times$ such that $\prod a_i=\prod b_i$, there exists a $q$-periodic function $f\in M_{K,q}$, unique up to multiplication by a constant, with zeros exactly $\cup a_i q^\ZZ$ and poles exactly $\cup b_i q^\ZZ$. Conversely, for any $f\in M_{K,q}$ with all of its zeros and poles in $K^\times$, we can find sets $\{a_1,\cdots,a_k\}$ and $\{b_1,\cdots, b_k\}$ in $K^\times$ such that the zeros are  $\cup a_i q^\ZZ$, the poles are $\cup b_i q^\ZZ$, and $\prod a_i=\prod b_i$.
\end{thm}
The proof can be found in Theorem~4.6 of~\cite{brown2003}. This theorem shows that any $q$-periodic meromorphic function $f$ can be expressed in terms of theta functions 
\begin{equation}\label{thetaexpformula}
f=c\frac{\prod_{i=1}^k\theta_{a_i}}{\prod_{i=1}^k\theta_{b_i}},    
\end{equation}
where all the coefficients $c, a_i$
and $b_i$ are in $K^\times$. All of these discussions upon Theorem~\ref{abeljacobithm} can be summarized in terms of divisors on $K^\times$.

\begin{prop}\label{divexactseq}
Fix $q$ with $|q|<1$. Denote by $\operatorname{Div}(E_q)$ the set of divisors on $E_q$,  by $\operatorname{\Omega}_q$ the set of nonzero theta functions on $E_q$, and by $\operatorname{\Omega}^0_q$ the set of theta functions in $E_q$ with no zeros or poles. Then the canonical map $f\to\operatorname{div}(f)$ induces an exact sequence of groups
\begin{equation}
1\longrightarrow \operatorname{\Omega}^0_q  \longrightarrow \operatorname{\Omega}_q
 \longrightarrow\operatorname{Div}(E_q)\longrightarrow 1.
\end{equation}
\end{prop}
\begin{proof}
The proof is similar to that of~\cite[Appendix I, Theorem~5.1]{hirzebruch1992}.    
\end{proof}
Furthermore, over the multiplicative group $M_{K,q}^\times$, we have the following exact sequence.
\begin{prop}\label{periodexactseq}
Let $\operatorname{Div}_0(E_q)$ denote the group of divisors of degree $0$ on $E_q$, and let $\operatorname{Div}_{\operatorname{p}}(E_q)$ denote its subgroup of principal divisors on $E_q$. Then  
$f\to\operatorname{div}(f)$ induces an exact sequence of groups
\begin{equation}
1\longrightarrow K^\times  \longrightarrow M_{K,q}^\times\longrightarrow\operatorname{Div}_{\operatorname{p}}(E_q)\longrightarrow 1,
\end{equation}
and we have 
\begin{equation}\label{chardivp}
\operatorname{Div}_{\operatorname{p}}(E_q)=\left\{   D\in \operatorname{Div}_0(E_q):D=\prod (Q)^{n_Q},\; \;
\prod  Q^{n_Q}\equiv1
\right\}.
\end{equation}
\end{prop}
\begin{proof}
The surjectivity is the first part of Theorem~\ref{abeljacobithm}. The exactness at the middle is a direct consequence of the intersection $H_K\cap M_{K,q}=\{\text{constants}\}$. Finally, the formula~\ref{chardivp}, as a characterization of $\operatorname{Div}_{\operatorname{p}}(E_q)$, is in essence a restatement of Theorem~\ref{abeljacobithm}. 
\end{proof}

Next we state the Riemann-Roch theorem on a Tate curve, which is much simpler than the standard Riemann-Roch theorem on a Riemann surface.
\begin{thm}[Riemann-Roch, Theorem~4.10 in~\cite{brown2003}]\label{rrthm}
Let $\mathcal L(n)$ denote the $K$-vector space of functions $f\in M_{K,q}$ such that $f$ has poles of degree$\leq n$ at $q^\ZZ$ and no other poles.
Then we have 
\begin{equation}
\operatorname{dim}_K(\mathcal L(n))= \begin{cases}
n  	& n\geq1\\ 
1 						& 
n=0.
\end{cases}    
\end{equation}
\end{thm}
\begin{proof}
This follows immediately from the following classification of dimensions
\begin{equation}\label{dimoftheta}
\operatorname{dim}_K(H_{K,q}(cu^r))= \begin{cases}
r  	& r>0\\ 
1 						& 
r=0, c\in q^\ZZ\\ 
0 & 
r=0, c\notin q^\ZZ; r<0. 
\end{cases}      
\end{equation}
\end{proof}

This theorem along with formula~\ref{thetaexpformula} provides a complete characterization of a $q$-periodic meromorphic function $f$ that only has poles at $q^\ZZ$.
\begin{rem}
The Riemann-Roch theorem on a Riemann surface can be reduced to $ \operatorname{dim}(\mathcal L(D))=\operatorname{deg}(D)-g+1$ for $\operatorname{deg}(D)>2g-2$. Here we work on an  $p$-adic version of elliptic curve, so $g=1$  yields $\operatorname{dim}(\mathcal L(D))=\operatorname{deg}(D)=n$, which is essentially what the above theorem says.   
\end{rem}

After all these preparations, we finally get to the point of discussing a general $l$-torsion point on $E_q/K$. 
In the following we work on a sufficiently large field $K$ that includes the primitive $l$-th root of unity and $\tilde q:=q^{\frac{1}{2l}}$. 
\begin{prop}\label{lthroottheta}
Let $l\geq3$ be a prime and $Q\in E_q[l]$ be any $l$-torsion point on $E_q$.
Then there exists a function $f_Q(u)$ with the following properties:
\begin{itemize}
\item[1.] The map $u\to f_Q(u)$ is a theta function with respect to $q^\ZZ$.
\item[2.] The divisor of  $f_Q(u)$ is $(1)/(Q)$.
\item[3.] $f_Q(u)\to(1-u)$ for $u\to1$, and for any $Q'=Qq^m, m\in\ZZ$, there exists an $l$-th root of unity $\mu$ such that $f_{Q'}(u)=\mu f_{Q}(u)$.
\end{itemize}
\end{prop}
\begin{proof}
By proposition~\ref{divexactseq}, the function $f_Q$ satisfying the first two conditions exists and is unique up to a multiple in $\Omega_q^0$. Hence  $f_Q^l$ is a theta function with divisor $(1)^l/(Q)^l$ and $1^l/Q^l=1$. By proposition~\ref{divexactseq} and Proposition~\ref{periodexactseq}, we see that $f_Q^l$ is a $q$-periodic meromorphic function up to a multiple in $\Omega_q^0$. Multiplying $f_Q$ by a constant, we can assume that $f_Q^l$ is a $q$-periodic meromorphic function with $f_Q(u)\to(1-u)$ for $u\to1$. Then $(f_{Q'}/ f_{Q})^l$ is a normalized $q$-periodic meromorphic function with no zeros or poles, so by Proposition~\ref{periodexactseq} we have $(f_{Q'}/ f_{Q})^l=1$.
\end{proof}

Next we show that the $l$-th root of the theta function yields $l$-th torsion points. Define a change of variable by $u=v^l$ and $\zeta(v):= \theta(u)^{\frac{1}{l}}$. Since the theta function satisfies 
$ \theta( u)^{\frac{1}{l}}=-u^{\frac{1}{l}}\theta(qu)^{\frac{1}{l}}$, we have $\zeta(v)=-v\zeta( q^{\frac{1}{l}}v)$. Hence by formula~\ref{dimoftheta}, $\zeta(v)=\theta(q^{\frac{1}{l}}, v)$ which implies that $\zeta$ is a theta function with the unique zero at $v=q^{\frac{1}{l}\ZZ}$. Note that $v^l\in q^{\ZZ}$, thus the reciprocal of $l$-th root of theta function $\zeta^{-1}$ is essentially the product of fractions $f_Q/\theta$ for some $f_Q$ given in Proposition~\ref{lthroottheta}, up to a multiplication by an $l$-th root of unity.    

Using relation~\ref{thetarelation}, we see that $\tilde\theta$ strictly corresponds to $\theta$, and 
their zeros and poles counted with multiplicities are exactly the same through this correspondence. 
Therefore, $\tilde\theta$ only vanishes at  $\tilde u^2=u=q^\ZZ$, and then the $l$-th root of $\tilde\theta$ vanishes
at $v=q^{\frac{j}{2l}}$ for $0\leq j<l^*= (l-1)/2$, where $j$ takes these integers due to $\tilde u\in q^{\frac{1}{2}\ZZ}$. 

In IUT~\cite[II]{mochizuki12}, the key construction that gives rise to the inequality of diophantine results is a reciprocal of the $l$-th root of the theta function given by
\begin{equation}
\Theta(u):= \left\{\left( \sqrt{-1}\sum_{n=-\infty}^{\infty}q^{\frac{1}{2}(n+\frac{1}{2})^2} \right)^{-1}\left( \sum_{n=-\infty}^{\infty}(-1)^n  q^{\frac{1}{2}(n+\frac{1}{2})^2} u^{n+\frac{1}{2}}\right)\right\}^{-\frac{1}{l}},
\end{equation}
where the term within parentheses can be seen as a normalization of $\tilde\theta$ at $-1$.
By the above discussion, $\Theta$ can be seen as a non-vanishing meromorphic function that only has poles at $q^{\frac{1}{2l}\ZZ}$. Then formula~\ref{qperiodic} of "$q$-periodicity" gives us the theta value, up to a factor of a $2l$-th root of unity, at $l$-torsion points $q^{\frac{a}{2l}}$ for $j\in\ZZ$,
\begin{equation}\label{Thetalink}
\Theta(\tilde q^{j})=\Theta(q^{\frac{j}{2l}})=(-1)^j q^{\frac{j^2}{2l}} (-1)^j \Theta(-1)=q^{\frac{j^2}{2l}}=\tilde q^{j^2}.    
\end{equation}
Here we use the fact that the normalized $\tilde\theta$ is $1$ at $-1$. Also note that $\Theta$ is well-defined on a Tate curve since $(j\pm2l)^2\equiv j^2$ mod $2l$. 
In addition, formula~\ref{qperiodic} gives us a 
modulo $\operatorname{log}(\mathcal O_K^\times)$-functional equation~\ref{functionaleq} with respect to $p$-adic logarithm as indicated in~\cite[Proposition~1.5]{mochizuki09}.
\begin{equation}\label{functionaleq}
\operatorname{log}(\Theta(q^{\frac{j}{2}}u))+\operatorname{log}(\mathcal O_K^\times) =\operatorname{log}(\Theta(u))+\frac{j}{l}\log(u)+\frac{j^2}{2l}\operatorname{log}(q)+\operatorname{log}(\mathcal O_K^\times).    
\end{equation}

\subsection{Hodge Theaters}
Now we have the information for both sides of the fibration~\ref{schemeexactseq3*}, it is natural to consider the interaction of the Frobenius picture $\operatorname{Diff}^{\operatorname{an}}(R_\mathfrak p)$  and the \'etale picture $\operatorname{Aut}(\widetilde{\mathcal E_{\mathfrak a}}[l])$, i.e., the interplay of $l$-torsion points of theta functions $\Theta$ in $K_{\mathfrak p}$-coefficient Tate curves and $p$-adic logarithmic functions. This two-parameter construction reminds us of a Hodge structure in many aspects, and this is why such a structure is called a $\Theta^{\pm \operatorname{ell}}$NF-Hodge theater in IUT~\cite[III Definition~3.8]{mochizuki12}.

The fundamental idea of constructing a Hodge structure out of fibration~\ref{schemeexactseq5} arises from the decomposition of  
$R^\times_{\mathfrak p}$ into two distinct parts
$R^{\times\mu}_{\mathfrak p}$ and $R^{\mu}_{\mathfrak p}$; see Example~\ref{logvolume}.
The residue part $R^{\times\mu}_{\mathfrak p}$ is contained in $1+\mathfrak p$, so we can apply $p$-adic logarithm to obtain a new element in $R^\times_{\mathfrak p}$, which gives the full log-links by iteration of this process.
\begin{equation}\label{loglink}
\cdots\stackrel{\operatorname{log}}{\longrightarrow}\mathcal{HT}_{n,m-1}^{\Theta^{\pm \operatorname{ell}}\operatorname{NF}}\stackrel{\operatorname{log}}{\longrightarrow}\mathcal{HT}_{n,m}^{\Theta^{\pm \operatorname{ell}}\operatorname{NF}}\stackrel{\operatorname{log}}{\longrightarrow}\mathcal{HT}_{n,m+1}^{\Theta^{\pm \operatorname{ell}}\operatorname{NF}}\stackrel{\operatorname{log}}{\longrightarrow}\cdots     
\end{equation}
On the other hand, the torsion part $R^{\mu}_{\mathfrak p}$ 
is related to the theta function $\Theta$ valued in the group $\mu_{2l}$ of $2l$-th roots of unity.
The explicit form of the group action on the torsion can be given in the same way as in formula~\ref{Thetalink}, which iteratively gives the full $\Theta$-links.
\begin{equation}\label{thetalink}
\cdots\stackrel{\Theta}{\longrightarrow}\mathcal{HT}_{n-1,m}^{\Theta^{\pm \operatorname{ell}}\operatorname{NF}}\stackrel{\Theta}{\longrightarrow}\mathcal{HT}_{n,m}^{\Theta^{\pm \operatorname{ell}}\operatorname{NF}}\stackrel{\Theta}{\longrightarrow}\mathcal{HT}_{n+1,m}^{\Theta^{\pm \operatorname{ell}}\operatorname{NF}}\stackrel{\Theta}{\longrightarrow}\cdots     
\end{equation}
If $\mathcal{HT}_{n,m}^{\Theta^{\pm \operatorname{ell}}\operatorname{NF}}$ is viewed as an $(n,m)$-th element of an infinite dimensional order matrix, then the log-links (resp. $\Theta$-links) are vertical arrows (resp. horizontal arrows) of columns (resp. rows) in the matrix. This yields a log-theta-lattice, which is often referred to as "Gaussian" in IUT,  since theta function $\Theta$ acts as a Gaussian distribution function on the $l$-torsion points.

\subsection{A remark on Diophantine results}
In the preceding subsections, we learn how to construct the IUT from a view point of automorphisms on group schemes. This seems to be a much simpler way to introduce the object, and it arises naturally as a $p$-adic counterpart of automorphisms on principal bundles of differential geometry. Now in the last part of present section, we give a brief remark on Mochizuki's idea of proving various Diophantine results including the celebrated $abc$-conjecture.

As mentioned several times in IUT, the core idea in proving the Diophantine results is a coordinate change from the Cartesian $(x,y)$ to the Polar $(r,\theta)$ in integration of Gaussian distribution, which is simply expressed as 
\begin{align*}
\left(\int_{-\infty}^{+\infty} e^{-x^2} dx \right)^2=\int_{-\infty}^{+\infty}\int_{-\infty}^{+\infty} e^{-(x^2+y^2)}dxdy=\int_{0}^{2\pi}\int_{0}^{+\infty} e^{-r^2}rdrd\theta=\pi.  
\end{align*}
If this idea is applied to a group scheme over a number field with fibre being the $l$-th torsion points of an elliptic curve $E$, then the coordinate change will be the one from the local Tate parameter $(q_v,u_v)$ to the local polar parameter 
$(q_v^{\times\mu},j)$, where $\tilde q_v^j$ is the rotation of the $l$-th torsion and $q_v^{\times\mu}$ is the quotient within $R_{v}^{\times\mu}$. 

So the remaining task is to calculate the so-called "volume" of logarithmic theta function $\operatorname{log}(\Theta)$ in these two coordinates. The calculation carried out in the polar parameter can be simplified to equation~1.5 in Scholze et. al.~\cite{scholze2018}. On the other hand, the calculation in the Tate parameter is a comparison on degrees of Arakelov divisors of $\tilde q$ and $\Theta$, and the following main claim is asserted to hold
\begin{equation}\label{pilotestimate}
-|\operatorname{log}(\tilde q)|\leq-|\operatorname{log}(\Theta)|.    
\end{equation}
However, as pointed out by subsection~2.2 in~\cite{scholze2018}, there is a substantial gap between concrete $q$-pilot ($\Theta$-pilot) objects and abstract pilot objects. It seems impossible to obtain formula~\ref{pilotestimate} directly from concrete $q$-pilot ($\Theta$-pilot) objects, which we illustrate in the following example. 
\begin{expl}[heuristic]
Let $E_q$ be an elliptic curve with Tate parameter $q$. The term $-|\operatorname{log}(\tilde q)|$ is essentially the degree of following Arakelov divisor
\begin{align*}
-\tilde q_E=-\sum_{v} v(\tilde q_v)\cdot v=-\frac{1}{2l}q_E,   
\end{align*}
where $v$ goes through all the valuations for which $E_q$ has a bad reduction. Similarly, $-|\operatorname{log}(\Theta)|$ is the degree of following Arakelov divisor
\begin{align*}
-\Theta_E(\tilde q_v^j)=-\sum_{v} v(\Theta_v(\tilde q_v^j))\cdot v=-\frac{j^2}{2l}q_E,   
\end{align*}
where $j\in\{0,1,\cdots, l^*\}$ and $l^*=(l-1)/2$. One immediately arising problem is that the degree of divisor $\Theta_E(\tilde q_v^j)$ should be in $\mathbb Q/\ZZ$, since $q_v^\ZZ$ is the unit element in Tate curve. The remaining quantitative inequality is even harder to be true. 
\end{expl}

\bibliographystyle{abbrv}

\begin{thebibliography}{10}

\bibitem{abbati1989}
M.~C.~Abbati, R.~Cirelli, A.~Mania, and P.~Michor.
\newblock The Lie group of
automorphisms of a principal bundle.
\newblock {\em JGP}, 6(2):215-235, 1989.

\bibitem{Albeverio2010}
S.~Albeverio, A.~Y.~Khrennikov, and V.~M.~Shelkovich.
\newblock {\em Theory of $p$-Adic Distributions:
Linear and Non-Linear Models}. Volume 370 of London Mathematical Society Lecture
Note Series. 
\newblock Cambridge University Press, 2010.

\bibitem{Amari2000}
S.~Amari and H.~Nagaoka.
\newblock {\em Methods of Information Geometry}. Translations of Mathematical Monographs, volume 191.
\newblock Oxford University Press, 2000.

\bibitem{arnold1966geometrie}
V.~I.~Arnold.
\newblock Sur la g{\'e}om{\'e}trie diff{\'e}rentielle des groupes de {Lie} de
  dimension infinie et ses applications {\`a} l'hydrodynamique des fluides
  parfaits.
\newblock {\em Annales de l'institut Fourier}, 16(1):319-361, 1966.

\bibitem{Arnold98}
V.~I.~Arnold, B.~Khesin. \newblock {\em Topological Methods in Hydrodynamics}. Applied Mathematical Sciences, volume 125.  Springer-Verlag, 1998.

\bibitem{bauer2021paper}
M.~Bauer, Y.~Lu and C.~Maor.
\newblock A Geometric View on the Generalized Proudman-Johnson and $r$-Hunter-Saxton Equations.
\newblock {\em Journal of Nonlinear Science}, 17, 2022.


\bibitem{bauer2023}
M.~Bauer, A.~Le Brigant, Y.~Lu and C.~Maor.
\newblock The $L^p$-Fisher-Rao metric and Amari-\u{C}encov $\alpha$-connections.
\newblock {\em Calculus of Variations and Partial Differential Equations}, 63(2):56, 2024.




\bibitem{bradley2025}
P.~E.~Bradley.
\newblock Diffusion operators on $p$-adic analytic manifolds.
\newblock {\em 			ArXiv:2510.22563},  2025.

\bibitem{brown2003}
A.~F.~Brown.
\newblock  $p$-adic Theta Functions and Tate Curves.
\newblock {\em Elliptic Curves, Modular Forms and Cryptography, Hindustan Book Agency}, 151-165, 2003.

\bibitem{gunning1967}
R.~C.~Gunning.
\newblock Special coordinate coverings of Riemann surfaces.
\newblock {\em Math. Ann.}, 170:67-86, 1967.

\bibitem{hirzebruch1992}
F.~Hirzebruch, T.~Berger and R.~Jung.
\newblock{\em Manifolds and Modular Forms}. \newblock  Springer Fachmedien Wiesbaden, 1992.

%\bibitem{harris1982}J.~Harris.\newblock Theta-characteristics on algebraic curves.\newblock {\em  Trans. Amer. Math. Soc.}, 271(2): 611–638, 1982.

\bibitem{hubbard2006}
J.~H.~Hubbard.
\newblock{\em Teichm\"{u}ller and Applications to Geometry, Topology, and Dynamics, Volume 1: Teichm\"{u}ller Theory}. \newblock  Matrix Editions, 2006.

\bibitem{igusa2002}
J.~Igusa.
\newblock {\em An Introduction to the Theory of Local Zeta Functions}.
\newblock volume 14 of
AMS/IP studies in advanced mathematics, American Mathematical Society, International Press, 2002. 

%%\bibitem{dowty2018}J. G.~Dowty.\newblock Chentsov's theorem for exponentical families.\newblock {\em Information geometry}, 1:117-135, 2018. 

\bibitem{katz1982}
N.~M.~Katz.
\newblock A conjecture in the arithmetic theory of
differential equations.
\newblock {\em Bulletin de la S. M. F.}, 110:203-239, 1982.


\bibitem{khesin2011}
B.~Khesin, J.~Lenells, G.~Misiolek and S.~C.~Preston.
\newblock Geometry of diffeomorphism groups,
complete integrability and optimal transport.
\newblock {\em Geometric and Functional Analysis}, 23:334-366, 2013.

\bibitem{koblitz1977}
N.~Koblitz.
\newblock {\em p-adic Numbers, p-adic Analysis, and Zeta Functions.}
\newblock  {Graduate
Texts in Mathematics 58}, Springer-Verlag, 1977.


\bibitem{Lu2023}
Y.~Lu.
\newblock The $L^p$-Fisher-Rao metric and  information geometry.
\newblock FSU PhD thesis, 2023.

\bibitem{Lu2025}
Y.~Lu.
\newblock The $L^p$-geometry  and its applications.
\newblock {\em 		ArXiv:2512.13051,}  2025.
%\bibitem{mccann1997}R.~J.~McCann.\newblock A convexity principle for interacting gases.\newblock {\em Advances in Mathematics}, 128(1):153-179, 1997.

\bibitem{mochizuki96}
S.~Mochizuki.
\newblock A Theory of Ordinary p-adic Curves.
\newblock {\em Publ. Res. Inst. Math. Sci.}, 32: 957-1151, 1996.

\bibitem{mochizuki99}
S.~Mochizuki.
\newblock {\em Foundations of p-adic Teichm\"{u}ller Theory.}
\newblock  {AMS/IP Studies
in Advanced Mathematics 11.}, American Mathematical Society, 1999.

\bibitem{mochizuki02}
S.~Mochizuki.
\newblock An Introduction to p-adic Teichm\"{u}ller Theory.
\newblock {\em Cohomologies
p-adiques et applications arithm\'etiques I}, 278: 1-49, 2002.

\bibitem{mochizuki09}
S.~Mochizuki.
\newblock The \'Etale Theta function and its Frobenioid-theoretic Manifestations.
\newblock {\em Publ. Res. Inst. Math. Sci.}, 45: 227-349, 2009.

\bibitem{mochizuki12}
S.~Mochizuki.
\newblock Inter-universal Teichm\"{u}ller Theory I-IV.
\newblock {\em RIMS Preprint 1758.}, 2012.

\bibitem{mochizuki15}
S.~Mochizuki.
\newblock Topics in Absolute Anabelian Geometry III: Global Reconstruction Algorithms.
\newblock {\em J. Math. Sci. Univ. Tokyo}, 22: 939-1156, 2015.

\bibitem{moser1965}
J.~Moser.
\newblock On the volume elements on a manifold.
\newblock {\em Trans. Amer. Math. Soc.}, 120(2): 286-294, 1965.

%\bibitem{mumford1971}D.~Mumford.\newblock Theta characteristics of an algebraic curve.\newblock {\em  Ann. Sci. Ecole Norm. Sup.}, 4(2): 181-192, 1971.

%\bibitem{mumford1972}D.~Mumford.\newblock An analytic construction of degenerating curves over complete local rings.\newblock {\em Composition Math}, 24: 129-174, 1972.


\bibitem{neeb2011}
K.~H.~Neeb.
\newblock Lie Groups of Bundle Automorphisms and Their Extensions.
\newblock {\em Developments and Trends in Infinite-Dimensional Lie Theory. Progress in Mathematics, vol 288.} 
Birkhäuser, 2011.



\bibitem{schneider2011}
P.~Schneider.
\newblock $p$-adic Lie groups. 
\newblock {\em Grundlehren der mathematischen Wissenschaften
344}, Springer, 2011.

\bibitem{scholze2018}
P.~Scholze and J.~Stix.
\newblock Why $abc$ is still a conjecture.
2018.

\bibitem{serre1965}
J.~P.~Serre.
\newblock Classifcation des vari\'eti\'es analytiques p-adiques compactes.
\newblock {\em Topology}, 3:409–412, 1965.

\bibitem{serre1992}
J.~P.~Serre.
\newblock Lie Algebras and Lie Groups.
\newblock {\em Lectures given at Harvard University.
Lecture Notes in Mathematics 1500}, Springer, 1992.


\bibitem{silverman1986}
J.~H.~Silverman.
\newblock 
{\em The Arithmetic of Elliptic Curves}. \newblock  Graduate Texts in Mathematics
106. Springer, 1986.

\bibitem{silverman1994}
J.~H.~Silverman.
\newblock 
{\em Advanced Topics in the Arithmetic of Elliptic Curves}. \newblock  Graduate Texts in Mathematics
151. Springer, 1994.

\bibitem{weil1982}
A.~Weil.
\newblock 
{\em Ad\`eles and Algebraic Groups}. \newblock  Progress in Mathematics 23. Birkh\"{a}user,
 1982.

\end{thebibliography}

\end{document}